\documentclass[12pt]{article}
\usepackage{ct}

\usepackage{tikz}
\usetikzlibrary{calc}
\usepackage{placeins}
\usepackage[labelformat=empty]{subfig}

\specs{13}{4 (2)}{2024}{}

\dateline{May 29, 2023}{Jun 7, 2024}{Sep 30, 2024}

\keywords{Tilings, aperiodic order, polyforms}

\MSC{05B45, 52C20, 05B50}

\title{A chiral aperiodic monotile}

\author[1]{David Smith}
\author[2]{Joseph Samuel Myers}
\author[3]{Craig S. Kaplan}
\author[4]{Chaim Goodman-Strauss}

\affil[1]{%
Yorkshire, U.K.

\email{ds.orangery@gmail.com}%
}

\affil[2]{%
Cambridge, U.K.

\email{jsm@polyomino.org.uk}%
}

\affil[3]{%
School of Computer Science, University of Waterloo, Waterloo, Ontario, Canada

\email{csk@uwaterloo.ca}%
}

\affil[4]{%
National Museum of Mathematics, New York, New York, U.S.A.

\email{chaimgoodmanstrauss@gmail.com}%
}

\newcommand{\fignum     } [1] {\ref{#1}}
\newcommand{\fig        } [1] {Figure~\fignum{#1}}
\graphicspath{{figures/}}

\newcommand{\vcoords}[2]{($sqrt(0.75)*(#1,0)+0.5*(0,#1)+(0,#2)$)}
\newcommand{\hcoords}[2]{($sqrt(12)*(#1,0)+sqrt(3)*(#2,0)+3*(0,#2)+sqrt(3)*(1,0)+(0,1)$)}
\newcommand{\markpt}[2]{%
  \filldraw \vcoords{#1}{#2} circle [radius=3pt]
}
\newcommand{\vctxt}[3]{%
  \draw \vcoords{#1}{#2} node {#3}
}

\newcommand{\rawtileA}[6]{%
  #1[#2shift={\hcoords{#4}{#5}},rotate=#3]
    \vcoords{0}{0} -- \vcoords{-2}{1} -- \vcoords{-2}{0} --
    \vcoords{0}{-2} -- \vcoords{1}{-2} -- \vcoords{2}{-4} --
    \vcoords{4}{-5} -- \vcoords{4}{-4} -- \vcoords{3}{-3} --
    \vcoords{4}{-2} -- \vcoords{3}{0} -- \vcoords{2}{0} --
    \vcoords{1}{1} -- cycle;
  \draw[shift={\hcoords{#4}{#5}},rotate=#3] \vcoords{2}{-1} node {#6}
}

\newcommand{\rawtileB}[6]{%
  #1[#2shift={\hcoords{#4}{#5}},rotate=#3]
    \vcoords{0}{0} -- \vcoords{-2}{1} -- \vcoords{-2}{0} --
    \vcoords{-3}{0} -- \vcoords{-2}{-2} -- \vcoords{2}{-4} --
    \vcoords{3}{-3} -- \vcoords{4}{-4} -- \vcoords{5}{-4} --
    \vcoords{4}{-2} -- \vcoords{2}{-1} -- \vcoords{2}{0} --
    \vcoords{1}{1} -- cycle;
  \draw[shift={\hcoords{#4}{#5}},rotate=#3] \vcoords{0}{-2} node {#6}
}

\newcommand{\colgrid}{lightgray}

\renewcommand{\markpt}[2]{%
  \filldraw \vcoords{#1}{#2} circle [radius=2.5pt]
}

\definecolor{lightergray}{rgb}{0.93, 0.93, 0.93}
\definecolor{mysticupper}{rgb}{0.73, 0.73, 0.58}
\definecolor{mysticlower}{rgb}{0.85, 0.86, 0.78}
\definecolor{deltacolor}{rgb}{0.64, 0.59, 0.52}
\definecolor{sigmacolor}{rgb}{0.75, 0.57, 0.49}
\definecolor{gammacolor}{rgb}{0.8, 0.62, 0.49}
\definecolor{xicolor}{rgb}{0.83, 0.69, 0.56} 
\definecolor{picolor}{rgb}{0.86, 0.77, 0.63} 
\definecolor{phicolor}{rgb}{0.89, 0.84, 0.66}
\definecolor{psicolor}{rgb}{0.88, 0.88, 0.61} 
\definecolor{thetacolor}{rgb}{0.82, 0.84, 0.59}
\definecolor{lambdacolor}{rgb}{0.72, 0.8, 0.7}

\newcommand{\hexsevenB}[1]{\rawtileB{\draw}{\colgrid,fill=mysticupper,}{0}{0}{0}{#1};
  \rawtileA{\draw}{\colgrid,}{120}{1}{0}{};
  \rawtileA{\draw}{\colgrid,}{180}{0}{1}{};
  \rawtileA{\draw}{\colgrid,}{60}{-1}{0}{};
  \rawtileA{\draw}{\colgrid,}{-60}{0}{-1}{};
  \rawtileA{\draw}{\colgrid,}{0}{1}{-1}{};
  \rawtileA{\draw}{\colgrid,}{60}{2}{-1}{};
  \draw[shift=\hcoords{0}{0}, line width = 1.5pt]
  \vcoords{-5}{1} -- \vcoords{-4}{0} -- \vcoords{-4}{-1} -- \vcoords{-4}{-2} -- \vcoords{-3}{-3} -- \vcoords{-4}{-4} -- \vcoords{-3}{-6} -- \vcoords{-2}{-6} -- \vcoords{-2}{-5} -- \vcoords{0}{-6} -- \vcoords{1}{-5} -- \vcoords{2}{-6} -- \vcoords{3}{-6} -- \vcoords{4}{-8} -- \vcoords{6}{-9} -- \vcoords{6}{-8} -- \vcoords{7}{-8} -- \vcoords{8}{-8} -- \vcoords{8}{-7} -- \vcoords{10}{-8} -- \vcoords{11}{-7} -- \vcoords{10}{-6} -- \vcoords{9}{-6} -- \vcoords{8}{-4} -- \vcoords{6}{-3} -- \vcoords{6}{-2} -- \vcoords{5}{-1} -- \vcoords{6}{0} -- \vcoords{5}{2} -- \vcoords{4}{2} -- \vcoords{3}{3} -- \vcoords{2}{4} -- \vcoords{1}{4} -- \vcoords{0}{6} -- \vcoords{-2}{7} -- \vcoords{-2}{6} -- \vcoords{-1}{5} -- \vcoords{-2}{4} -- \vcoords{-4}{5} -- \vcoords{-4}{4} -- \vcoords{-5}{4} -- \vcoords{-4}{2}
  -- cycle}
  \newcommand{\hexeightB}[1]{\rawtileB{\draw}{\colgrid,fill=mysticupper,}{0}{0}{0}{#1};
  \rawtileA{\draw}{\colgrid,}{120}{1}{0}{};
  \rawtileA{\draw}{\colgrid,}{180}{0}{1}{};
  \rawtileA{\draw}{\colgrid,fill=mysticlower,}{60}{-1}{0}{};
  \rawtileA{\draw}{\colgrid,}{-60}{0}{-1}{};
  \rawtileA{\draw}{\colgrid,}{0}{1}{-1}{};
  \rawtileA{\draw}{\colgrid,}{0}{1}{-2}{};
  \rawtileA{\draw}{\colgrid,}{60}{2}{-1}{};
  \draw[shift=\hcoords{0}{0},  line width = 1.5pt]
  \vcoords{-5}{1} -- \vcoords{-4}{0} -- \vcoords{-4}{-1} -- \vcoords{-4}{-2} -- \vcoords{-3}{-3} -- \vcoords{-4}{-4} -- \vcoords{-3}{-6} -- \vcoords{-2}{-6} -- \vcoords{-1}{-7} -- \vcoords{0}{-8} -- \vcoords{1}{-8} -- \vcoords{2}{-10} -- \vcoords{4}{-11} -- \vcoords{4}{-10} -- \vcoords{3}{-9} -- \vcoords{4}{-8} -- \vcoords{6}{-9} -- \vcoords{6}{-8} -- \vcoords{7}{-8} -- \vcoords{8}{-8} -- \vcoords{8}{-7} -- \vcoords{10}{-8} -- \vcoords{11}{-7} -- \vcoords{10}{-6} -- \vcoords{9}{-6} -- \vcoords{8}{-4} -- \vcoords{6}{-3} -- \vcoords{6}{-2} -- \vcoords{5}{-1} -- \vcoords{6}{0} -- \vcoords{5}{2} -- \vcoords{4}{2} -- \vcoords{3}{3} -- \vcoords{2}{4} -- \vcoords{1}{4} -- \vcoords{0}{6} -- \vcoords{-2}{7} -- \vcoords{-2}{6} -- \vcoords{-1}{5} -- \vcoords{-2}{4} -- \vcoords{-4}{5} -- \vcoords{-4}{4} -- \vcoords{-5}{4} -- \vcoords{-4}{2}
  -- cycle}
  \newcommand{\hexnineB}[1]{\rawtileB{\draw}{\colgrid,fill=mysticupper,}{0}{0}{0}{#1};
  \rawtileA{\draw}{\colgrid,}{120}{1}{0}{};
  \rawtileA{\draw}{\colgrid,}{180}{0}{1}{};
  \rawtileA{\draw}{\colgrid,fill=mysticlower,}{60}{-1}{0}{};
  \rawtileA{\draw}{\colgrid,}{-60}{0}{-1}{};
  \rawtileA{\draw}{\colgrid,}{0}{1}{-1}{};
  \rawtileA{\draw}{\colgrid,}{0}{1}{-2}{};
  \rawtileA{\draw}{\colgrid,}{60}{2}{-1}{};
  \rawtileA{\draw}{\colgrid,}{120}{2}{0}{};
  \draw[shift=\hcoords{0}{0}, line width = 1.5pt]
  \vcoords{-5}{1} -- \vcoords{-4}{0} -- \vcoords{-4}{-1} -- \vcoords{-4}{-2} -- \vcoords{-3}{-3} -- \vcoords{-4}{-4} -- \vcoords{-3}{-6} -- \vcoords{-2}{-6} -- \vcoords{-1}{-7} -- \vcoords{0}{-8} -- \vcoords{1}{-8} -- \vcoords{2}{-10} -- \vcoords{4}{-11} -- \vcoords{4}{-10} -- \vcoords{3}{-9} -- \vcoords{4}{-8} -- \vcoords{6}{-9} -- \vcoords{6}{-8} -- \vcoords{7}{-8} -- \vcoords{8}{-8} -- \vcoords{8}{-7} -- \vcoords{10}{-8} -- \vcoords{11}{-7} -- \vcoords{10}{-6} -- \vcoords{10}{-5} -- \vcoords{10}{-4} -- \vcoords{9}{-3} -- \vcoords{10}{-2} -- \vcoords{9}{0} -- \vcoords{8}{0} -- \vcoords{8}{-1} -- \vcoords{6}{0} -- \vcoords{5}{2} -- \vcoords{4}{2} -- \vcoords{3}{3} -- \vcoords{2}{4} -- \vcoords{1}{4} -- \vcoords{0}{6} -- \vcoords{-2}{7} -- \vcoords{-2}{6} -- \vcoords{-1}{5} -- \vcoords{-2}{4} -- \vcoords{-4}{5} -- \vcoords{-4}{4} -- \vcoords{-5}{4} -- \vcoords{-4}{2}
  -- cycle}
\newcommand{\edgealphaplus}[4]{%
  \draw[shift={\hcoords{#3}{#4}},rotate=#1] [rotate=#2] \vcoords{-2}{0} -- \vcoords{-2}{2};
  \fill[shift={\hcoords{#3}{#4}},rotate=#1] [rotate=#2] \vcoords{-2}{0.75} -- \vcoords{-2}{1.25} -- \vcoords{-2.5}{1.25} -- cycle
}
\newcommand{\edgealphaminus}[4]{%
  \draw[shift={\hcoords{#3}{#4}},rotate=#1] [rotate=#2] \vcoords{-2}{0} -- \vcoords{-2}{2};
  \fill[shift={\hcoords{#3}{#4}},rotate=#1] [rotate=#2] \vcoords{-2}{0.75} -- \vcoords{-2}{1.25} -- \vcoords{-1.5}{0.75} -- cycle
}
\newcommand{\edgebetaplus}[4]{%
  \draw[shift={\hcoords{#3}{#4}},rotate=#1] [rotate=#2] \vcoords{-2}{0} -- \vcoords{-2}{2};
  \fill[shift={\hcoords{#3}{#4}},rotate=#1] [rotate=#2] \vcoords{-2}{1.25} arc[start angle=90,delta angle=180,radius=0.25] -- cycle
}
\newcommand{\edgebetaminus}[4]{%
  \draw[shift={\hcoords{#3}{#4}},rotate=#1] [rotate=#2] \vcoords{-2}{0} -- \vcoords{-2}{2};
  \fill[shift={\hcoords{#3}{#4}},rotate=#1] [rotate=#2] \vcoords{-2}{0.75} arc[start angle=-90,delta angle=180,radius=0.25] -- cycle
}
\newcommand{\edgegammaplus}[4]{%
  \draw[shift={\hcoords{#3}{#4}},rotate=#1] [rotate=#2] \vcoords{-2}{0} -- \vcoords{-2}{2};
  \draw[shift={\hcoords{#3}{#4}},rotate=#1] [rotate=#2] \vcoords{-2}{0.75} -- \vcoords{-2}{1.25} -- \vcoords{-2.5}{1.25} -- cycle
}
\newcommand{\edgegammaminus}[4]{%
  \draw[shift={\hcoords{#3}{#4}},rotate=#1] [rotate=#2] \vcoords{-2}{0} -- \vcoords{-2}{2};
  \draw[shift={\hcoords{#3}{#4}},rotate=#1] [rotate=#2] \vcoords{-2}{0.75} -- \vcoords{-2}{1.25} -- \vcoords{-1.5}{0.75} -- cycle
}
\newcommand{\edgedeltaplus}[4]{%
  \draw[shift={\hcoords{#3}{#4}},rotate=#1] [rotate=#2] \vcoords{-2}{0} -- \vcoords{-2}{2};
  \draw[shift={\hcoords{#3}{#4}},rotate=#1] [rotate=#2] \vcoords{-2}{1.25} -- \vcoords{-2.5}{1.5};
  \draw[shift={\hcoords{#3}{#4}},rotate=#1] [rotate=#2] \vcoords{-2}{0.75} -- \vcoords{-2.5}{1}
}
\newcommand{\edgedeltaminus}[4]{%
  \draw[shift={\hcoords{#3}{#4}},rotate=#1] [rotate=#2] \vcoords{-2}{0} -- \vcoords{-2}{2};
  \draw[shift={\hcoords{#3}{#4}},rotate=#1] [rotate=#2] \vcoords{-2}{1.25} -- \vcoords{-1.5}{1};
  \draw[shift={\hcoords{#3}{#4}},rotate=#1] [rotate=#2] \vcoords{-2}{0.75} -- \vcoords{-1.5}{0.5}
}
\newcommand{\edgeepsilonplus}[4]{%
  \draw[shift={\hcoords{#3}{#4}},rotate=#1] [rotate=#2] \vcoords{-2}{0} -- \vcoords{-2}{2};
  \draw[shift={\hcoords{#3}{#4}},rotate=#1] [rotate=#2] \vcoords{-2}{1} -- \vcoords{-2.5}{1.25}
}
\newcommand{\edgeepsilonminus}[4]{%
  \draw[shift={\hcoords{#3}{#4}},rotate=#1] [rotate=#2] \vcoords{-2}{0} -- \vcoords{-2}{2};
  \draw[shift={\hcoords{#3}{#4}},rotate=#1] [rotate=#2] \vcoords{-2}{1} -- \vcoords{-1.5}{0.75}
}
\newcommand{\edgezetaplus}[4]{%
  \draw[shift={\hcoords{#3}{#4}},rotate=#1] [rotate=#2] \vcoords{-2}{0} -- \vcoords{-2}{2};
  \draw[shift={\hcoords{#3}{#4}},rotate=#1] [rotate=#2] \vcoords{-2}{1.25} -- \vcoords{-2.5}{1.5};
  \draw[shift={\hcoords{#3}{#4}},rotate=#1] [rotate=#2] \vcoords{-2}{1} -- \vcoords{-2.5}{1.25};
  \draw[shift={\hcoords{#3}{#4}},rotate=#1] [rotate=#2] \vcoords{-2}{0.75} -- \vcoords{-2.5}{1}
}
\newcommand{\edgezetaminus}[4]{%
  \draw[shift={\hcoords{#3}{#4}},rotate=#1] [rotate=#2] \vcoords{-2}{0} -- \vcoords{-2}{2};
  \draw[shift={\hcoords{#3}{#4}},rotate=#1] [rotate=#2] \vcoords{-2}{1.25} -- \vcoords{-1.5}{1};
  \draw[shift={\hcoords{#3}{#4}},rotate=#1] [rotate=#2] \vcoords{-2}{1} -- \vcoords{-1.5}{0.75};
  \draw[shift={\hcoords{#3}{#4}},rotate=#1] [rotate=#2] \vcoords{-2}{0.75} -- \vcoords{-1.5}{0.5}
}
\newcommand{\edgeeta}[4]{%
  \draw[shift={\hcoords{#3}{#4}},rotate=#1] [rotate=#2] \vcoords{-2}{0} -- \vcoords{-2}{2}
}
\newcommand{\edgethetaplus}[4]{%
  \draw[shift={\hcoords{#3}{#4}},rotate=#1] [rotate=#2] \vcoords{-2}{0} -- \vcoords{-2}{2};
  \draw[shift={\hcoords{#3}{#4}},rotate=#1] [rotate=#2] \vcoords{-2}{1.25} arc[start angle=90,delta angle=180,radius=0.25] -- cycle
}
\newcommand{\edgethetaminus}[4]{%
  \draw[shift={\hcoords{#3}{#4}},rotate=#1] [rotate=#2] \vcoords{-2}{0} -- \vcoords{-2}{2};
  \draw[shift={\hcoords{#3}{#4}},rotate=#1] [rotate=#2] \vcoords{-2}{0.75} arc[start angle=-90,delta angle=180,radius=0.25] -- cycle
}
\newcommand{\hexGamma}[3]{%
\fill[gammacolor,shift={\hcoords{#2}{#3}},rotate=#1] \vcoords{-2}{2} --
    \vcoords{-2}{0} -- \vcoords{0}{-2} -- \vcoords{2}{-2} --
    \vcoords{2}{0} -- \vcoords{0}{2} -- cycle;
  \edgealphaminus{#1}{0}{#2}{#3};
  \edgealphaplus{#1}{60}{#2}{#3};
  \edgegammaminus{#1}{120}{#2}{#3};
  \edgedeltaminus{#1}{180}{#2}{#3};
  \edgebetaplus{#1}{240}{#2}{#3};
  \edgebetaminus{#1}{300}{#2}{#3};
  \draw[shift={\hcoords{#2}{#3}},rotate=#1] \vcoords{0}{0} node[transform shape] {$\Gamma$}
}
\newcommand{\hexDelta}[3]{%
  \fill[deltacolor,shift={\hcoords{#2}{#3}},rotate=#1] \vcoords{-2}{2} --
    \vcoords{-2}{0} -- \vcoords{0}{-2} -- \vcoords{2}{-2} --
    \vcoords{2}{0} -- \vcoords{0}{2} -- cycle;
  \edgegammaplus{#1}{0}{#2}{#3};
  \edgebetaplus{#1}{60}{#2}{#3};
  \edgeepsilonminus{#1}{120}{#2}{#3};
  \edgealphaplus{#1}{180}{#2}{#3};
  \edgegammaminus{#1}{240}{#2}{#3};
  \edgezetaminus{#1}{300}{#2}{#3};
  \draw[shift={\hcoords{#2}{#3}},rotate=#1] \vcoords{0}{0} node[transform shape] {$\Delta$}
}
\newcommand{\hexTheta}[3]{%
  \fill[thetacolor,shift={\hcoords{#2}{#3}},rotate=#1] \vcoords{-2}{2} --
    \vcoords{-2}{0} -- \vcoords{0}{-2} -- \vcoords{2}{-2} --
    \vcoords{2}{0} -- \vcoords{0}{2} -- cycle;
  \edgegammaplus{#1}{0}{#2}{#3};
  \edgebetaplus{#1}{60}{#2}{#3};
  \edgethetaplus{#1}{120}{#2}{#3};
  \edgebetaplus{#1}{180}{#2}{#3};
  \edgeeta{#1}{240}{#2}{#3};
  \edgebetaminus{#1}{300}{#2}{#3};
  \draw[shift={\hcoords{#2}{#3}},rotate=#1] \vcoords{0}{0} node[transform shape] {$\underline{\Theta}$}
}
\newcommand{\hexLambda}[3]{%
  \fill[lambdacolor,shift={\hcoords{#2}{#3}},rotate=#1] \vcoords{-2}{2} --
    \vcoords{-2}{0} -- \vcoords{0}{-2} -- \vcoords{2}{-2} --
    \vcoords{2}{0} -- \vcoords{0}{2} -- cycle;
  \edgegammaplus{#1}{0}{#2}{#3};
  \edgebetaplus{#1}{60}{#2}{#3};
  \edgeepsilonminus{#1}{120}{#2}{#3};
  \edgealphaplus{#1}{180}{#2}{#3};
  \edgethetaminus{#1}{240}{#2}{#3};
  \edgebetaminus{#1}{300}{#2}{#3};
  \draw[shift={\hcoords{#2}{#3}},rotate=#1] \vcoords{0}{0} node[transform shape] {$\Lambda$}
}
\newcommand{\hexXi}[3]{%
  \fill[xicolor,shift={\hcoords{#2}{#3}},rotate=#1] \vcoords{-2}{2} --
    \vcoords{-2}{0} -- \vcoords{0}{-2} -- \vcoords{2}{-2} --
    \vcoords{2}{0} -- \vcoords{0}{2} -- cycle;
  \edgealphaminus{#1}{0}{#2}{#3};
  \edgeepsilonplus{#1}{60}{#2}{#3};
  \edgethetaplus{#1}{120}{#2}{#3};
  \edgebetaplus{#1}{180}{#2}{#3};
  \edgeeta{#1}{240}{#2}{#3};
  \edgebetaminus{#1}{300}{#2}{#3};
  \draw[shift={\hcoords{#2}{#3}},rotate=#1] \vcoords{0}{0} node[transform shape] {$\underline{\Xi}$}
}
\newcommand{\hexPi}[3]{%
  \fill[picolor,shift={\hcoords{#2}{#3}},rotate=#1] \vcoords{-2}{2} --
    \vcoords{-2}{0} -- \vcoords{0}{-2} -- \vcoords{2}{-2} --
    \vcoords{2}{0} -- \vcoords{0}{2} -- cycle;
  \edgealphaminus{#1}{0}{#2}{#3};
  \edgeepsilonplus{#1}{60}{#2}{#3};
  \edgeepsilonminus{#1}{120}{#2}{#3};
  \edgealphaplus{#1}{180}{#2}{#3};
  \edgethetaminus{#1}{240}{#2}{#3};
  \edgebetaminus{#1}{300}{#2}{#3};
  \draw[shift={\hcoords{#2}{#3}},rotate=#1] \vcoords{0}{0} node[transform shape] {$\Pi$}
}
\newcommand{\hexSigma}[3]{%
  \fill[sigmacolor,shift={\hcoords{#2}{#3}},rotate=#1] \vcoords{-2}{2} --
    \vcoords{-2}{0} -- \vcoords{0}{-2} -- \vcoords{2}{-2} --
    \vcoords{2}{0} -- \vcoords{0}{2} -- cycle;
  \edgezetaplus{#1}{0}{#2}{#3};
  \edgebetaplus{#1}{60}{#2}{#3};
  \edgeepsilonminus{#1}{120}{#2}{#3};
  \edgealphaplus{#1}{180}{#2}{#3};
  \edgegammaminus{#1}{240}{#2}{#3};
  \edgedeltaplus{#1}{300}{#2}{#3};
  \draw[shift={\hcoords{#2}{#3}},rotate=#1] \vcoords{0}{0} node[transform shape] {$\Sigma$}
}
\newcommand{\hexPhi}[3]{%
  \fill[phicolor,shift={\hcoords{#2}{#3}},rotate=#1] \vcoords{-2}{2} --
    \vcoords{-2}{0} -- \vcoords{0}{-2} -- \vcoords{2}{-2} --
    \vcoords{2}{0} -- \vcoords{0}{2} -- cycle;
  \edgegammaplus{#1}{0}{#2}{#3};
  \edgebetaplus{#1}{60}{#2}{#3};
  \edgeepsilonminus{#1}{120}{#2}{#3};
  \edgeepsilonplus{#1}{180}{#2}{#3};
  \edgeeta{#1}{240}{#2}{#3};
  \edgebetaminus{#1}{300}{#2}{#3};
  \draw[shift={\hcoords{#2}{#3}},rotate=#1] \vcoords{0}{0} node[transform shape] {$\underline{\Phi}$}
}
\newcommand{\hexPsi}[3]{%
  \fill[psicolor,shift={\hcoords{#2}{#3}},rotate=#1] \vcoords{-2}{2} --
    \vcoords{-2}{0} -- \vcoords{0}{-2} -- \vcoords{2}{-2} --
    \vcoords{2}{0} -- \vcoords{0}{2} -- cycle;
  \edgealphaminus{#1}{0}{#2}{#3};
  \edgeepsilonplus{#1}{60}{#2}{#3};
  \edgeepsilonminus{#1}{120}{#2}{#3};
  \edgeepsilonplus{#1}{180}{#2}{#3};
  \edgeeta{#1}{240}{#2}{#3};
  \edgebetaminus{#1}{300}{#2}{#3};
  \draw[shift={\hcoords{#2}{#3}},rotate=#1] \vcoords{0}{0} node[transform shape] {$\Psi$}
}

\extrafloats{100}
\begin{document}

\maketitle


\begin{abstract}
	The recently discovered ``hat'' aperiodic monotile mixes unreflected
	and reflected tiles in every tiling it admits, leaving open the question
	of whether a single shape can tile aperiodically using translations
	and rotations alone.  We show that a close relative of the hat---the 
	equilateral member of the continuum to which it belongs---is
	a weakly chiral aperiodic monotile: it admits only non-periodic 
	tilings if we forbid reflections by fiat.  Furthermore, by modifying
	this polygon's edges we obtain a family of shapes called Spectres
	that are strictly chiral aperiodic monotiles: they admit only homochiral
	non-periodic tilings based on a hierarchical substitution system.
\end{abstract}



\section{Introduction}
\label{sec:intro}

The recently discovered ``hat'' aperiodic monotile admits tilings
of the plane, but none that are periodic~\cite{hat}.  This 
polygon settles the question of whether a single 
shape---a closed topological disk in the plane---can
tile aperiodically without any additional matching
rules or other constraints on tile placement.

The hat is asymmetric: it is not equal to its image under any
non-trivial isometry of the plane.  In particular, a hat cannot
be brought into perfect correspondence with its own mirror reflections. 
Moreover, all tilings formed by copies of the hat must 
use both unreflected and reflected tiles.  
Some people have wondered whether the hat and its reflection
ought to be considered two distinct shapes, thereby invalidating 
its status as a monotile.  To some extent, this question is about 
tiles as physical objects rather than mathematical abstractions.
A hat cut from paper or plastic can easily be turned over
in three dimensions to obtain its reflection, but a glazed ceramic
tile cannot.  More broadly, a wide range of three-dimensional objects,
from organic molecules to shoes, behave very differently in their left-
and right-handed forms.

Since Felix Klein introduced the {\em Erlangen} program~\cite{greenberg},
each metric geometry has been defined by first choosing a set 
of isometries, the rigid
motions of the geometric space in which we are working.  
Isometries may be classified as \emph{orientation-preserving}, which
map left-handed shapes to left-handed shapes (and right-handed to
right-handed), or
\emph{orientation-reversing}, which exchange left- and right-handedness.
In the plane, the orientation-preserving isometries comprise translations
and rotations.  We will refer to all orientation-reversing isometries
generically as ``reflections'', and the image of a shape under a reflection
as a ``reflected'' copy of that shape.

From that point of view, the question of whether a single tile admits
monohedral tilings of the plane depends upon the set of 
isometries that may be used to transform tiles.
If this set is not specified, one may reasonably assume that the full
set of planar isometries is intended.
Gr\"unbaum and Shephard explicitly permit reflections in their definition
of monohedral tilings~\cite[Section 1.2]{GS}, so that when they later
ask for a tile that ``only admits monohedral non-periodic 
tilings''~\cite[Section 10.7]{GS}, we can be confident that they 
considered reflected tiles to be in play.
Ammann's aperiodic pairs of tiles were considered congruent under
reflections~\cite{AGS}, and the disconnected aperiodic Taylor--Socolar
tile~\cite{ST1,ST2} also requires reflections to tile the plane.
Likewise, we regard the hat as an aperiodic monotile.
Still, the fact that every tiling by hats mixes unreflected and
reflected tiles leaves open the question of whether
there might exist a monotile that achieves aperiodicity without the use
of reflections.  

A \emph{tile} $T$ is a closed topological disk in the plane, and a
\emph{monohedral} \emph{tiling} admitted by it is a countable collection
$\mathcal{T} = \{T_1, T_2,\ldots \}$ of congruent copies of $T$ with
disjoint interiors, whose
union is the entire plane.  Each $T_i$ is of the form $g_iT$ for
some planar isometry $g_i$.  We say that a monohedral tiling
$\mathcal{T}$ is a \emph{homochiral tiling} if for
every pair $T_i, T_j \in \mathcal{T}$, there is an orientation-preserving
isometry mapping $T_i$ to $T_j$.  (Note that a monohedral tiling whose
tile has bilateral symmetry is automatically homochiral according to this
definition.)
We then define a \emph{weakly chiral
aperiodic monotile} to be a tile whose homochiral tilings are all
non-periodic (and that admits at least one such tiling), and a
\emph{strictly chiral aperiodic monotile} to be a tile that admits
\emph{only} homochiral non-periodic tilings.  Following Klein, the weak
case is aperiodic if we decree reflections to be off limits, even
if the tile admits periodic tilings when reflections are allowed.
The strict case remains both chiral and aperiodic in the presence
of reflections.  With these definitions in hand, we ask: \emph{Do
there exist any weakly or strictly chiral aperiodic monotiles?}\footnote{We
might call this the ``vampire einstein'' problem, as we are seeking
a shape that is not accompanied by its reflection.}

The discovery of the hat is an effective reminder of how little
we understand the possibilities and subtleties of monohedral
tilings.  Certainly there is no evidence to suggest that the hat
(and the continuum of shapes to which it belongs) is somehow unique,
and we might therefore hope that a zoo of interesting new monotiles
will emerge in its wake.  In this context, a chiral aperiodic
monotile does not seem particularly unlikely.
Nonetheless, we did not expect to find one so close at hand:
a solution follows in short order from the hat.

\begin{figure}
\begin{center}
	\includegraphics[width=.9\textwidth]{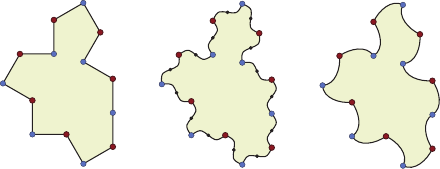}
\end{center}
\caption{ \label{fig:polygonspectre} 
	The $14$-sided polygon $\mathrm{Tile}(1,1)$, shown on the left, 
	is a weakly chiral aperiodic monotile: if by fiat we forbid tilings
	that mix unreflected and reflected tiles, then it admits only
	non-periodic tilings. By modifying its edges, as shown in the 
	centre and right for example, we obtain strictly chiral aperiodic
	monotiles called ``Spectres'' that admit only non-periodic 
	tilings even when reflections are permitted.}
\end{figure}

\begin{figure}
\begin{center}
	\includegraphics[width=\textwidth]{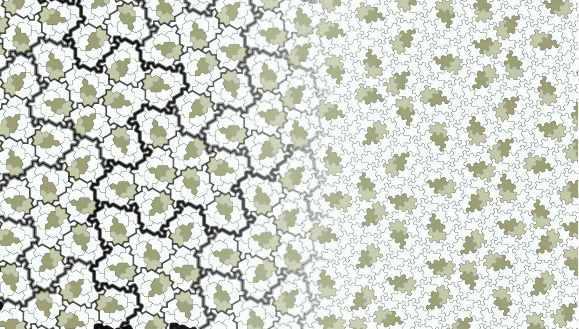}
\end{center}
\caption{ \label{fig:sampler} 
	A patch from a non-periodic tiling by Spectres.  On the left,
	tiles are drawn as copies of $\mathrm{Tile}(1,1)$, and thickened
	outlines show the hierarchy of supertiles to which these tiles
	will be shown to belong.
	On the right, tile boundaries are modified in a manner similar to 
	Figure~\ref{fig:polygonspectre} (right).  Tile colours will be
	explained later in the paper.}
\end{figure}

Our previous work~\cite{hat} related the discovery of two separate
aperiodic polykites---the hat, and a $10$-kite known as the
``turtle''.  We showed that these two shapes are members of a
continuum parameterized by choices of two non-negative edge lengths
$a$ and $b$. Each member of that continuum is a $14$-sided polygon (with
two collinear sides) denoted
$\mathrm{Tile}(a,b)$; the hat is $\mathrm{Tile}(1,\sqrt{3})$ and
the turtle is $\mathrm{Tile}(\sqrt{3},1)$.  All members of this
continuum are aperiodic monotiles, with three exceptions: $\mathrm{Tile}(0,1)$
(the ``chevron''), $\mathrm{Tile}(1,0)$ (the ``comet''), and the
equilateral polygon $\mathrm{Tile}(1,1)$ (Figure~\ref{fig:polygonspectre}, 
left).

Because $\mathrm{Tile}(1,1)$ is equilateral, copies of the tile may
fit together in more ways than generic members of the continuum,
and in fact it admits a simple periodic tiling using equal numbers
of left- and right-handed tiles~\cite[Figure 6.2]{hat}.  But what
happens if we rein in that freedom slightly?  In this paper we prove
that $\mathrm{Tile}(1,1)$ is a weakly chiral aperiodic monotile: if
by fiat we restrict ourselves to tilings using only translations
and rotations, then $\mathrm{Tile}(1,1)$ admits only non-periodic
tilings.  More importantly, by modifying the edges of this polygon
(Figure~\ref{fig:polygonspectre}, centre and right), we define
 a family of strictly chiral aperiodic monotiles that we call ``Spectres'',
which admit only homochiral tilings, even when reflections
are permitted (Theorem~\ref{maintheorem}).
Figure~\ref{fig:sampler} shows non-periodic tilings by 
$\mathrm{Tile}(1,1)$ (left) and a Spectre (right).

\section{The Spectre and its tilings}
\label{sec:construction}

Our main results concern Spectres, the set of shapes whose tilings
correspond exactly to the homochiral tilings admitted by
$\mathrm{Tile}(1,1)$,
thereby boosting us from weakly chiral to strictly chiral aperiodicity. 
It is helpful
to begin by giving a precise definition of this set.

We regard $\mathrm{Tile}(1,1)$ as an
equilateral polygon with $14$ unit-length edges and $14$ vertices, where one of 
those vertices lies between two collinear edges.  A tile $X$ is a
\emph{Spectre} if and only if

\begin{itemize}
	\item $X$ admits only homochiral tilings;
	\item Every tiling admitted by $X$ corresponds to one by
		$\mathrm{Tile}(1,1)$: if $\{g_iX\}$ is a tiling for a set of 
		isometries $\{g_i\}$, then $\{g_i\mathrm{Tile}(1,1)\}$ is also a 
		tiling;
	\item Every homochiral tiling admitted by $\mathrm{Tile}(1,1)$ corresponds
		to one by $X$: if $\{g_i\mathrm{Tile}(1,1)\}$ is a homochiral tiling,
		then $\{g_iX\}$ is also a tiling; and
	\item If $\{g_iX\}$ is a tiling, then that tiling and 
		$\{g_i\mathrm{Tile}(1,1)\}$
		have the same tiling vertices (points shared by three or more
		tiles).
\end{itemize}
    
We do not attempt to characterize the space of all Spectres, but 
we can assert that the space is non-empty.

\begin{lemma} 
\label{existlemma}
There exists a Spectre.\end{lemma}

\begin{proof}
We present two explicit constructions that yield families of Spectres, 
where the first is a special case of the second.

Note first that any tiling by $\mathrm{Tile(1,1)}$ is
\emph{vertex-to-vertex}: no vertex of a tile can lie in the interior
of an edge of an adjacent tile (where we explicitly separate the
long side of $\mathrm{Tile}(1,1)$ into two short edges, introducing
a vertex with an interior angle of $180^\circ$).  By considering
angles at consecutive vertices of the tile, we can see that any
maximal line segment in the union of the boundaries of the tiles
in a tiling has length $1$ or~$2$ (and in particular is finite),
forcing the edges on either side of that segment to be aligned.

We call a simple  planar path an \emph{s-curve} if it is
symmetric under $180^\circ$ rotation about its centre.  Choose a
smooth, non-straight s-curve $\sigma$, and construct a new tile $S$ by
replacing every edge of~$\mathrm{Tile}(1,1)$ by a translated and
rotated copy of $\sigma$, fit to the edge's vertices.
As shown in Figure~\ref{fig:polygonspectre} (centre), we can choose
$\sigma$ in such a way that we obtain a topological disk $S$.
Because the path is smooth, an argument similar to the one
above guarantees that every tiling by $S$ will be aligned along 
whole copies of $\sigma$, and therefore correspond to a tiling by 
$\mathrm{Tile}(1,1)$ with the same tiling vertices.  Furthermore, 
the symmetry of $\sigma$ prevents 
$S$ and its reflection from being adjacent in a tiling, 
meaning that the only possible
tilings by $S$ correspond to homochiral tilings by $\mathrm{Tile}(1,1)$.
Thus $S$ is a Spectre.

The s-curves above are in some sense the minimally invasive way to
``chiralize'' a polygon that admits vertex-to-vertex tilings, 
preserving as many of its legal adjacencies
as possible while forcing homochirality.  However, we can take
advantage of additional structure in $\mathrm{Tile}(1,1)$ to offer a
more general construction.  Observe that the interior angles at
the vertices of $\mathrm{Tile}(1,1)$ strictly alternate between
multiples of~$90^\circ$ and multiples of~$120^\circ$ (coloured blue
and red, respectively, in Figure~\ref{fig:polygonspectre}).  In any
tiling by $\mathrm{Tile}(1,1)$, colours must agree where vertices
of neighbouring tiles meet.  Thus we can choose \emph{any} (simple,
 non-straight, smooth) path~$\gamma$ and replace the edges
of~$\mathrm{Tile}(1,1)$ by copies of~$\gamma$, in such a way that
copies associated with consecutive edges of the tile are related
by a rotation around the vertex between them (so that the paths
alternately face ``in'' and ``out'').  As before, provided
that $\gamma$~is chosen so that this process yields a topological
disk, as in Figure~\ref{fig:polygonspectre} (right), that shape
will be a Spectre.  This construction subsumes the first one, because
the two possible ways to affix an s-curve to an edge yield identical
results.
\end{proof}

From this point on, we will not take the trouble to distinguish between 
$\mathrm{Tile}(1,1)$ restricted by fiat to homochiral tilings, and the
Spectres defined above, which only admit homochiral tilings---we know that
we are working with the same universe of tilings in either case. 
For simplicity we will usually speak of ``the'' Spectre in reference to any
of these shapes, and we will use the polygon $\mathrm{Tile}(1,1)$ in 
figures showing patches of Spectres.  However, if we refer to the Spectre 
in the context of strictly chiral aperiodicity, it should be understood that
we are excluding $\mathrm{Tile}(1,1)$.

\subsection{Main result}

When a substitution tiling is aperiodic, the proof of aperiodicity
often relies in some way on showing that the tiles are nested within
an infinite hierarchical superstructure in every tiling they admit.
We say that a set of tiles is \emph{hierarchical} if, in every
tiling admitted by those tiles, every tile is nested within an
infinite hierarchy of ever-larger supertiles~\cite{aht}.  If these
hierarchies are uniquely determined, then the tilings that contain
them must be non-periodic~\cite[Theorem~10.1.1]{GS}.  (For suppose
a hierarchical tiling $\mathcal{T}$ is periodic, meaning that there exists
a non-zero vector $\mathbf{v}$ for which 
$\mathcal{T}+\mathbf{v}=\mathcal{T}$.
Then there would exist some supertile $\mathcal{C}$ that is so large that
it has a non-empty intersection with $\mathcal{C}+\mathbf{v}$.
Every tile $T$ that lies in that intersection belongs to two distinct
supertiles at the same level of the hierarchy, contradicting the
uniqueness of that hierarchy.)
This approach to aperiodicity informs our main result.

\begin{theorem}\label{maintheorem} The Spectre admits a tiling, and in 
	any tiling it admits, each tile is contained within an infinite,
	unique hierarchy of larger and larger supertiles.  Thus the
	Spectre is a strictly chiral aperiodic monotile.
\end{theorem}

\begin{figure}[t]
\begin{center}
	\includegraphics[width=\textwidth]{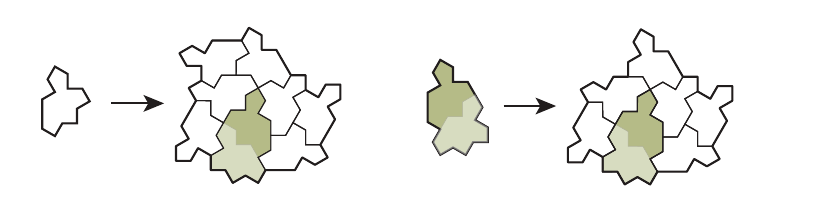}
\end{center}
\caption{\label{fig:thespectretiles}
	Two substitution rules that can be iterated to tile the plane
 	with Spectres.  The rules are based on a single Spectre and a 
	two-Spectre compound called a Mystic.  The first rule~(left)
	replaces the Spectre by a cluster containing a Mystic and seven
	Spectres, all reflected;
	the second~(right) replaces a Mystic by a cluster containing
	a Mystic and six Spectres.  In Section~\ref{sec:hexagons} we show that 
	every tiling by Spectres can be composed into non-overlapping
	congruent copies of these clusters.}
\end{figure}

\begin{figure}[ht!]
\centerline{
\includegraphics[width=.9\textwidth]{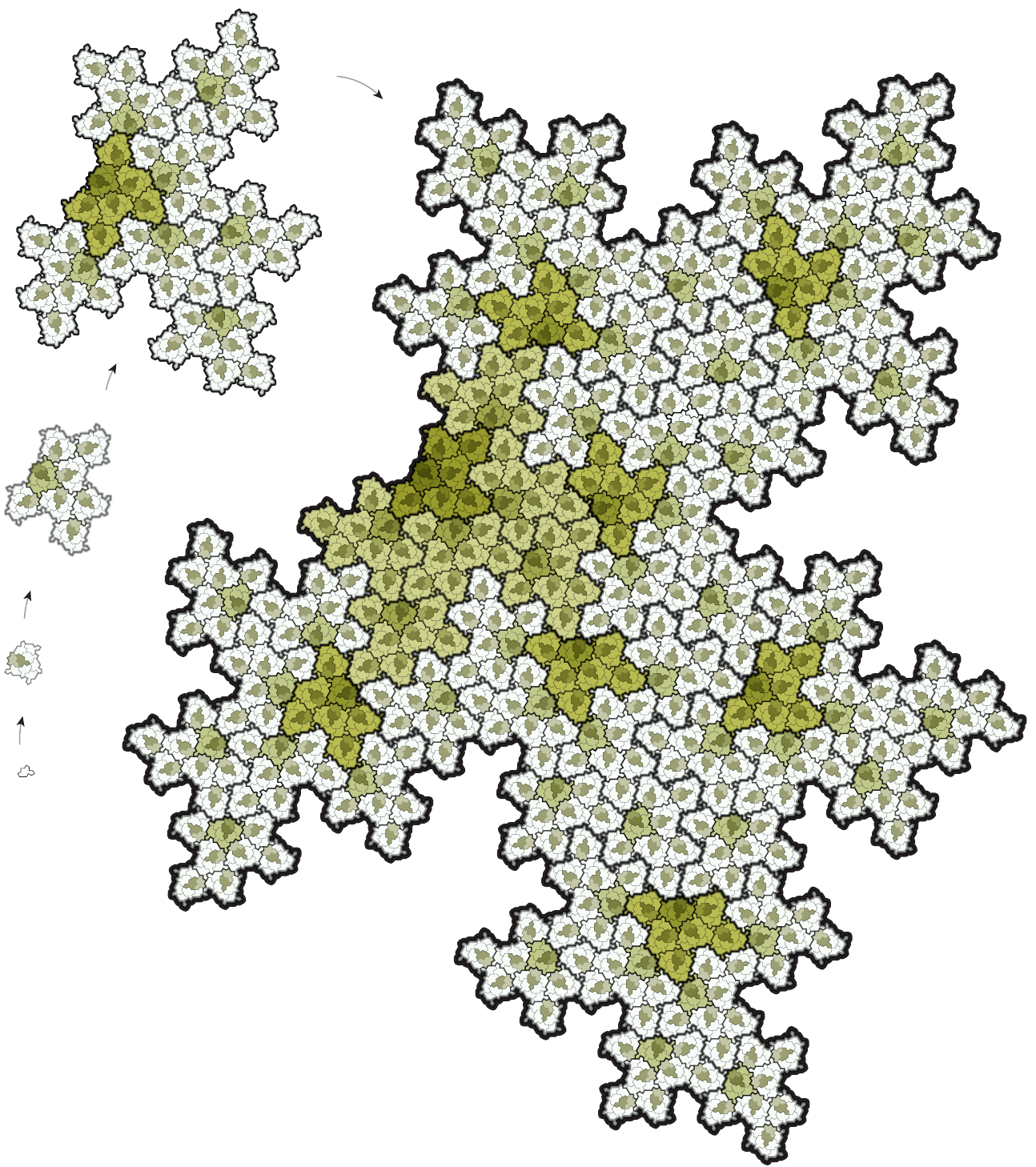}}
\caption{\label{fig:superdupertiles} Five levels of supertiles created
by applying the substitution rules given in Figure~\ref{fig:thespectretiles},
starting with a single Spectre.  The Mystic and its supertiles
are shaded in progressive levels of green.  
Each step uses Spectres of the opposite handedness of the one before it.}
\end{figure}

The first step in proving this theorem is to show that in any tiling
by Spectres, it is possible to partition the tiles into congruent
copies of the clusters of eight and nine Spectres shown in
Figure~\ref{fig:thespectretiles}.
In Section~\ref{sec:hexagons} we give a computer-assisted proof of
this fact.  We further show that these clusters belong to
nine distinct categories based on the relative positions of
neighbouring clusters (Figure~\ref{fig:clusters}), and that each
category may be viewed combinatorially as a regular hexagon marked
with matching rules in the form of labelled
edges (Figure~\ref{fig:clustershex}).  The proof relies on exhaustive
enumeration of patches of Spectres that can occur in tilings.  At
the outset, such an enumeration appears complicated: unlike hats,
Spectres are not polyforms, which suggests that generation of patches
cannot be reduced to discrete computations on an underlying grid.
But as we show in Section~\ref{sec:hatturtle}, any 
tiling by Spectres is combinatorially equivalent to a tiling by
hats and turtles~\cite{hat}, allowing us to perform the computations
we need with minor modifications to battle-tested software.

In Section~\ref{sec:hexsubst} we prove that in any tiling by the
nine marked hexagons described above, the hexagons can be uniquely
composed into supertiles that are combinatorially equivalent to
reflections of those hexagons.  The marked hexagons are therefore
hierarchical.  
We then observe that we can ``reverse'' the
supertile composition to obtain well-defined substitution rules on
the marked hexagons.  These rules can be iterated to produce a patch
of hexagons of any size.  It follows that the hexagons tile the
plane, and from any such tiling we may construct a
combinatorially equivalent tiling of the plane by Spectres, meaning
that the Spectre tiles the plane as well.  Because all such tilings
were previously shown to be non-periodic, the Spectre is a strictly
chiral aperiodic monotile.

The marked hexagons of Figure~\ref{fig:clustershex}, and the
substitution rules defined for them in Section~\ref{sec:hexsubst},
suffice to establish the aperiodicity of the Spectre.  However, it
is also possible to begin this process one step earlier, and define
substitution rules at the level of Spectres themselves.  These rules
are illustrated in Figure~\ref{fig:thespectretiles}.  Here we operate
on two base tiles: the first is a single Spectre, and the second is
a symmetric two-Spectre compound that we call a ``Mystic'' because
of its resemblance to a Buddha in seated meditation.  The first
substitution rule replaces a single Spectre with a ``Spectre cluster''
containing a Mystic and seven Spectres; the second replaces a Mystic
by a ``Mystic cluster'' containing a Mystic and six Spectres.
At the end of 
Section~\ref{sec:hexsubst} we transcribe the combinatorics of our
marked hexagons down onto the Spectre and Mystic, showing that each
tile in a tiling by Spectres lies in a unique infinite hierarchy of
supertiles.

Based on the substitution rules in Figure~\ref{fig:thespectretiles},
and the fact that the higher-order rules on hexagons are combinatorially
equivalent, we can use an eigenvalue computation to deduce properties
of Spectre tilings.  We find that in the limit, areas increase by
a factor of~$4+\sqrt{15}$ with each stage of the substitution, and
the ratio of the number of ``even'' Spectres (white and light green
in the figure) to the number of ``odd'' Spectres (dark green) is
also~$4+\sqrt{15}$.  Even and odd tiles are discussed in greater
detail in Section~\ref{sec:hatturtle}.  The irrationality of this
number immediately implies that any tiling produced using the
substitution rules must be non-periodic.  Nevertheless, we need the
full proof in Sections~\ref{sec:hexagons} and~\ref{sec:hexsubst}
to be certain that the substitution rules account for all possible
Spectre tilings.

At multiple steps in our proof, we transform a tiling into another
tiling by different tiles, and use properties of the new tiling to 
reason about the old one.  These transformations are deterministic and unique.
It is also important that they ``preserve periodicity'', that is, that
the original tiling is periodic if and only if the transformed tiling is.
In some situations, such as when grouping tiles into supertiles, this 
preservation is immediate: the two tilings have the same symmetry group.
When working through a less rigid combinatorial equivalence of tilings,
we require that when pairs of vertices correspond, then either both
pairs are related by translational symmetries of their tilings, or neither
is (Section~\ref{sec:combeq}).  This preservation of translational
symmetries allows us to conclude that every homochiral tiling by
$\mathrm{Tile}(1,1)$, and every tiling by Spectres must
be non-periodic.  

Remarkably, the rules of Figure~\ref{fig:thespectretiles} reverse
all tile orientations.  Consequently, in any sequence of iterations
of these rules, we will necessarily produce homochiral patches of
alternating handedness, as illustrated in Figure~\ref{fig:superdupertiles}.
We have never previously encountered this phenomenon when working with
substitution tilings.
Of course, the alternation can be avoided by defining
larger rules that combine two rounds of substitution.

The proofs used in this work are highly combinatorial in nature,
which sidesteps the practical question of how to draw patches of
Spectres algorithmically.  In Appendix~\ref{spectresubstAppendix}
we include a simple algorithm for laying out Spectres, derived from
the substitution rules in Figure~\ref{fig:thespectretiles}.  The algorithm
can be used to draw large patches like the ones in 
Figure~\ref{fig:superdupertiles}.  We have created an interactive
browser-based visualization tool for drawing patches of Spectres, 
available at~\href{https://cs.uwaterloo.ca/~csk/spectre/}{\nolinkurl{cs.uwaterloo.ca/~csk/spectre/}}.

\subsection{Combinatorial equivalence of tilings}
\label{sec:combeq}

At many points in this paper, we deduce information about a tiling
by noting its combinatorial equivalence to some other tiling.  Before
proceeding we consider the notion of combinatorial equivalence in
detail.  Two patches, or two tilings,  are \emph{combinatorially
equivalent} if and only if they are  homeomorphic as topological
complexes. Two sets of tiles are \emph{combinatorially equivalent}
if each tiling admitted by one is combinatorially equivalent to a
tiling admitted by the other. 

The following lemma gives a test for establishing combinatorial
equivalence.  An \emph{edge patch} is a collection of tiles with
disjoint interiors, such that there exists a closed arc~$e$ that
is a connected component of the intersection of two of the tiles,
all the other tiles contain an endpoint of~$e$, and~$e$ lies in the
interior of the union of the tiles.
A \emph{vertex} of an edge
patch is a point in the interior of the union of the tiles that 
is shared by at least three tiles of the patch.

\begin{lemma}
\label{lemma:CombinatorialTilingLemma}
Let $\mathcal A$ and $\mathcal B$ be finite sets of tiles (possibly
equipped with matching rules) such that each edge patch of one is
combinatorially equivalent to some edge patch admitted by the other.
Suppose further that the combinatorial equivalence satisfies the local
consistency condition that if two edge patches from one set both have
vertices surrounded by the same arrangement of tiles (up to an
isometry mapping one vertex onto the other), then the corresponding
vertices in the corresponding edge patches from the other set are also
surrounded by the same arrangement of tiles.  
Then $\mathcal A$ and $\mathcal B$ are combinatorially equivalent.
\end{lemma}

\begin{proof} Any tiling admitted by $\mathcal A$ may be considered  as an abstract  complex with labelled vertices, edges, faces and incidences between them, but no further structure. Upon this complex we 
may chart a global geometry,  applying the local geometry provided by the  edge patches admitted by~$\mathcal B$.
 From this structure we obtain a geometric complex
 with the same combinatorics as the
tiling we started with. This complex is a 
simply-connected, complete,  Riemannian metric manifold of constant zero
curvature, and is therefore the Euclidean plane~\cite{hopf}.  Thus
we have described this complex as a tiling by tiles in $\mathcal B$ that is combinatorially equivalent to our initial tiling by tiles in $\mathcal A$.
 \end{proof}

Note that combinatorial equivalence preserves periodicity.
If $\mathcal{T}_\mathcal{A}$ and $\mathcal{T}_\mathcal{B}$ are 
combinatorially equivalent tilings
by $\mathcal{A}$ and $\mathcal{B}$, then 
$\mathcal{T}_\mathcal{A}$ is periodic if and only if
$\mathcal{T}_\mathcal{B}$ is.
If $\mathcal{T}_\mathcal{A}$ has as a symmetry a translation by
a vector~$\mathbf{v}$, and if two vertices $x_1$,~$x_2$ separated
by~$\mathbf{v}$ map to vertices of $\mathcal{T}_\mathcal{B}$ separated
by a vector~$\mathbf{w}$, then any other two vertices $y_1$,~$y_2$ of
$\mathcal{T}_\mathcal{A}$ separated by~$\mathbf{v}$ also map to
vertices separated by~$\mathbf{w}$, by considering corresponding
sequences of edge patches forming paths from $x_1$ to~$y_1$ and $x_2$
to~$y_2$.  Thus $\mathbf{w}$ defines a translational symmetry
of $\mathcal{T}_\mathcal{B}$.

\section{From Spectres to hats and turtles}
\label{sec:hatturtle}

Our computational analysis of the
hat~\cite{hat} was simplified by the fact that it is a polyform, specifically a union
of eight kites from the Laves tiling $[3.4.6.4]$.  Furthermore, the 
tiles in every tiling by hats must be aligned with the kites of the
Laves tiling~\cite[Lemma A.6]{hat}.  Consequently, patches of hats
can be manipulated discretely by associating information with the 
cells of the underlying kite grid.  The aperiodic $10$-kite known as the
turtle is also compatible with grid-based computations.

$\mathrm{Tile}(1,1)$ is not a polyform, which at the outset appears
to rule out such an approach.  However, we can regain the ability to 
perform discrete computations by exploiting a connection between 
tilings by $\mathrm{Tile}(1,1)$ and tilings by combinations of
hats and turtles.  We prove the following result.

\pagebreak 
\begin{theorem}
\label{thm:hatturtle}
There is a bijection between combinatorially equivalent tilings by
$\mathrm{Tile}(1,1)$ and by the set $\{\mbox{hat},\mbox{turtle}\}$,
such that a $\mathrm{Tile}(1,1)$ tiling has a translation as a symmetry
if and only if the corresponding hat-turtle tiling has a corresponding
translation, and the $\mathrm{Tile}(1,1)$ tiling includes a reflected 
tile if and only if the hat-turtle tiling does.
\end{theorem}

\begin{figure}[t]
\centerline{\includegraphics[width = \textwidth]{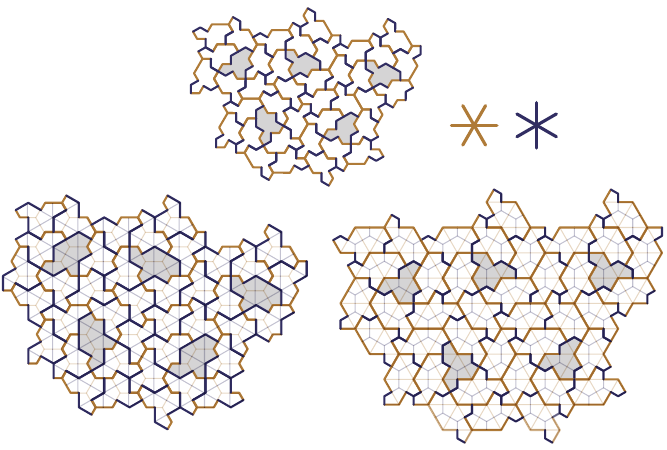}}
\caption{Tilings by $\mathrm{Tile}(1,1)$ are combinatorially equivalent 
to tilings by hats and turtles, through a change to the lengths
of the ``even'' and ``odd'' tile edges, oriented at even and odd multiples
of $30^\circ$ to the horizontal (top right). All three patches are
drawn so that short edges have unit length.}\label{figure:hatturtlespectre} 
\end{figure}

\begin{proof}
Note first that in any tiling by a combination of hats and turtles, the
tiles must be aligned to an underlying $[3.4.6.4]$ tiling.  This fact
can be established using an argument similar to the one used by Smith
et~al.\ for monohedral tilings by hats alone~\cite[Lemma A.6]{hat}.

Recall that the hat and the turtle are members of a continuum of tiles
parameterized by two non-negative edge lengths~$a$ and~$b$.  Each such
tile is denoted $\mathrm{Tile}(a,b)$; the hat is $\mathrm{Tile}(1,\sqrt{3})$
and the turtle is $\mathrm{Tile}(\sqrt{3},1)$.  

Let $\mathcal{T}$ be a tiling by $\mathrm{Tile}(1,1)$.  The edges
of the tiles of $\mathcal{T}$ lie in six families of parallel lines,
spaced evenly by multiples of $30^\circ$.  For convenience, we
rotate the tiling so that one of those families is horizontal.  In
that case, the edge vectors of tiles all form angles of $30k^\circ$
with the horizontal; we refer to edges as ``even'' or ``odd''
when $k$ is even or odd, respectively.  Furthermore, every tile in
$\mathcal{T}$ must appear in one of 24 orientations, twelve unreflected
and twelve reflected; arbitrarily choosing the tile's two collinear
edges as a reference, we label each tile as ``even'' or ``odd''
according to the parity of those edges. 
Figure~\ref{figure:hatturtlespectre} (top) shows a sample patch by
$\mathrm{Tile}(1,1)$, with the odd tiles shaded. Next to the patch 
are two stars showing the
even and odd edge directions.

Smith et~al.\ show that the two edge lengths in a tiling by
$\mathrm{Tile}(a,b)$ can be modified continuously to produce a
family of combinatorially equivalent monotiles~\cite[Theorem~6.1]{hat}.  Analogous modifications are possible here.  Around any
tile in $\mathcal{T}$, the edge vectors belonging
to any one direction add up to zero.  We can therefore change the
lengths of all edges in each of those six directions independently
of the others, provided that no tile boundary becomes
self-intersecting~\cite[Lemma~1.1]{regprod}.  In particular, if we
change the lengths of the even edges of~$\mathrm{Tile}(1,1)$
to~$\sqrt{3}$, then even tiles
become turtles and odd tiles become hats, all aligned to a 
single~$[3.4.6.4]$ Laves tiling (Figure~\ref{figure:hatturtlespectre}, bottom
right).  If instead we change the lengths of 
the odd edges to~$\sqrt{3}$, then even tiles become hats and odd tiles
become turtles (Figure~\ref{figure:hatturtlespectre}, bottom
left).
Similarly, if we begin with any tiling by hats and turtles and 
change all edge lengths of~$\sqrt{3}$ to~$1$, we obtain a tiling by
$\mathrm{Tile}(1,1)$.  We have created an animation that visualizes
these continuous edge length modifications on a sample patch of 
tiles---see \href{https://youtu.be/K6wXQvL5KRo}{\nolinkurl{youtu.be/K6wXQvL5KRo}}.

Thus there is a direct relationship between a tiling by
$\mathrm{Tile}(1,1)$ and a combinatorially equivalent tiling by
hats and turtles.  Orientations of tiles in the two tilings correspond
in the obvious way.  Unreflected and reflected copies of $\mathrm{Tile}(1,1)$
correspond to unreflected and reflected hats and turtles.
As in Lemma~\ref{lemma:CombinatorialTilingLemma} the
effect of this correspondence on a translation symmetry is given
by applying the correspondence to two vertices of one tiling that
are related by that translation.  
\end{proof}

The next section relies on the enumeration of patches of tiles that can
occur in homochiral tilings by $\mathrm{Tile}(1,1)$.  While we could in 
principle perform such calculations directly on a boundary representation
of the tile, we instead use the correspondence established here, 
constructing (and illustrating) patches of unreflected hats and turtles
as a proxy
for even and odd copies of~$\mathrm{Tile}(1,1)$.  By doing so, we can
harness convenient discrete algorithms similar to the ones we used 
previously to study the hat~\cite{Myers, Kaplan}, knowing that any
results we obtain apply to~$\mathrm{Tile}(1,1)$ as well.

\FloatBarrier


\section{From hats and turtles to marked hexagons}
\label{sec:hexagons}

We have not yet shown that the Spectre admits any tilings of the plane.
However, we can still prove that any such tilings, if they exist, must
be non-periodic.  This section and the one that follows furnish such
a proof.  As a by-product we also obtain a substitution system that 
can produce patches of Spectres of any size.

Any tiling by Spectres is combinatorially equivalent to a homochiral
tiling by $\mathrm{Tile}(1,1)$.  In turn, that tiling is equivalent
to a homochiral tiling by hats and turtles (i.e., all tiles are unreflected,
or all tiles are reflected).  These equivalences extend to any
translational symmetries of the tilings, meaning that the Spectre tiling
is periodic if and only if the hat-turtle tiling is.  As with the analysis
of hat tilings~\cite[Section~4]{hat}, we show here that in any homochiral 
tiling by hats and turtles, we can group tiles into non-overlapping 
clusters.  The resulting tiling by the clusters satisfies certain 
matching rules and has the same symmetries as the original tiling
by hats and turtles. 
Recalling that  two tilings are combinatorially equivalent if and only 
if they are homeomorphic as topological complexes, we arrive at the 
following result.

\begin{theorem}
\label{thm:clusters89}
Any tiling by Spectres is equivalent via Theorem~\ref{thm:hatturtle}
to a homochiral tiling by hats and turtles.  That tiling
can be composed into the clusters
shown in Figure~\ref{fig:clusters}, satisfying the matching rules
included there.  The tiling by those clusters has the same 
symmetries as the tiling by hats and turtles, and is 
combinatorially equivalent to a tiling by regular hexagons.
\end{theorem}

\begin{figure}[ht!]
\captionsetup{margin=0pt}%
\begin{center}
\subfloat[Cluster $\Gamma$]{%
\begin{tikzpicture}[x=1.95mm,y=1.95mm]
  \hexeightB{};
  \markpt{-3}{4};
  \markpt{-2}{-4};
  \markpt{0}{6};
  \markpt{5}{3};
  \markpt{6}{-11};
  \markpt{13}{-7};
  \vctxt{-3.5}{0}{$\alpha^-$};
  \vctxt{0.5}{-8}{$\alpha^+$};
  \vctxt{-2}{6}{$\beta^-$};
  \vctxt{3.5}{5.5}{$\beta^+$};
  \vctxt{10}{-9.5}{$\gamma^-$};
  \vctxt{9.5}{-2}{$\delta^-$};
\end{tikzpicture}%
} \qquad \subfloat[Cluster $\Delta$]{%
\begin{tikzpicture}[x=1.95mm,y=1.95mm]
  \hexnineB{};
  \markpt{-3}{4};
  \markpt{0}{7};
  \markpt{1}{-7};
  \markpt{6}{-10};
  \markpt{11}{0};
  \markpt{12}{-8};
  \vctxt{-3.5}{0}{$\gamma^+$};
  \vctxt{-2}{6}{$\zeta^-$};
  \vctxt{3.5}{-11}{$\beta^+$};
  \vctxt{10}{-9.5}{$\epsilon^-$};
  \vctxt{5.5}{4.5}{$\gamma^-$};
  \vctxt{13.5}{-4}{$\alpha^+$};
\end{tikzpicture}%
} \qquad \subfloat[Cluster $\Theta$]{%
\begin{tikzpicture}[x=1.95mm,y=1.95mm]
  \hexnineB{};
  \markpt{-3}{4};
  \markpt{0}{6};
  \markpt{1}{-7};
  \markpt{6}{-10};
  \markpt{10}{0};
  \markpt{12}{-5};
  \vctxt{-3.5}{0}{$\gamma^+$};
  \vctxt{-2}{6}{$\beta^-$};
  \vctxt{3.5}{-11}{$\beta^+$};
  \vctxt{10}{-9.5}{$\theta^+$};
  \vctxt{5}{4}{$\eta$};
  \vctxt{13.5}{-3}{$\beta^+$};
\end{tikzpicture}%
} \\ \subfloat[Cluster $\Lambda$]{%
\begin{tikzpicture}[x=1.95mm,y=1.95mm]
  \hexnineB{};
  \markpt{-3}{4};
  \markpt{0}{6};
  \markpt{1}{-7};
  \markpt{6}{-10};
  \markpt{11}{0};
  \markpt{12}{-8};
  \vctxt{-3.5}{0}{$\gamma^+$};
  \vctxt{-2}{6}{$\beta^-$};
  \vctxt{3.5}{-11}{$\beta^+$};
  \vctxt{10}{-9.5}{$\epsilon^-$};
  \vctxt{5.5}{4.5}{$\theta^-$};
  \vctxt{13.5}{-4}{$\alpha^+$};
\end{tikzpicture}%
} \qquad \subfloat[Cluster $\Xi$]{%
\begin{tikzpicture}[x=1.95mm,y=1.95mm]
  \hexnineB{};
  \markpt{-3}{4};
  \markpt{-2}{-4};
  \markpt{0}{6};
  \markpt{6}{-10};
  \markpt{10}{0};
  \markpt{12}{-5};
  \vctxt{-3.5}{0}{$\alpha^-$};
  \vctxt{2}{-9}{$\epsilon^+$};
  \vctxt{-2}{6}{$\beta^-$};
  \vctxt{10}{-9.5}{$\theta^+$};
  \vctxt{5}{4}{$\eta$};
  \vctxt{13.5}{-3}{$\beta^+$};
\end{tikzpicture}%
} \qquad \subfloat[Cluster $\Pi$]{%
\begin{tikzpicture}[x=1.95mm,y=1.95mm]
  \hexnineB{};
  \markpt{-3}{4};
  \markpt{-2}{-4};
  \markpt{0}{6};
  \markpt{6}{-10};
  \markpt{11}{0};
  \markpt{12}{-8};
  \vctxt{-3.5}{0}{$\alpha^-$};
  \vctxt{2}{-9}{$\epsilon^+$};
  \vctxt{-2}{6}{$\beta^-$};
  \vctxt{10}{-9.5}{$\epsilon^-$};
  \vctxt{5.5}{4.5}{$\theta^-$};
  \vctxt{13.5}{-4}{$\alpha^+$};
\end{tikzpicture}%
} \\ \subfloat[Cluster $\Sigma$]{%
\begin{tikzpicture}[x=1.95mm,y=1.95mm]
  \hexnineB{};
  \markpt{-2}{-1};
  \markpt{0}{7};
  \markpt{1}{-7};
  \markpt{6}{-10};
  \markpt{11}{0};
  \markpt{12}{-8};
  \vctxt{-3}{-3}{$\zeta^+$};
  \vctxt{-3.5}{5}{$\delta^+$};
  \vctxt{3.5}{-11}{$\beta^+$};
  \vctxt{10}{-9.5}{$\epsilon^-$};
  \vctxt{5.5}{4.5}{$\gamma^-$};
  \vctxt{13.5}{-4}{$\alpha^+$};
\end{tikzpicture}%
} \qquad \subfloat[Cluster $\Phi$]{%
\begin{tikzpicture}[x=1.95mm,y=1.95mm]
  \hexnineB{};
  \markpt{-3}{4};
  \markpt{0}{6};
  \markpt{1}{-7};
  \markpt{6}{-10};
  \markpt{10}{0};
  \markpt{12}{-8};
  \vctxt{-3.5}{0}{$\gamma^+$};
  \vctxt{-2}{6}{$\beta^-$};
  \vctxt{3.5}{-11}{$\beta^+$};
  \vctxt{10}{-9.5}{$\epsilon^-$};
  \vctxt{5}{4}{$\eta$};
  \vctxt{13.5}{-4}{$\epsilon^+$};
\end{tikzpicture}%
} \qquad \subfloat[Cluster $\Psi$]{%
\begin{tikzpicture}[x=1.95mm,y=1.95mm]
  \hexnineB{};
  \markpt{-3}{4};
  \markpt{-2}{-4};
  \markpt{0}{6};
  \markpt{6}{-10};
  \markpt{10}{0};
  \markpt{12}{-8};
  \vctxt{-3.5}{0}{$\alpha^-$};
  \vctxt{2}{-9}{$\epsilon^+$};
  \vctxt{-2}{6}{$\beta^-$};
  \vctxt{10}{-9.5}{$\epsilon^-$};
  \vctxt{5}{4}{$\eta$};
  \vctxt{13.5}{-4}{$\epsilon^+$};
\end{tikzpicture}
}%
\end{center}
\caption{The nine marked clusters of hats and turtles.  These clusters are combinatorially equivalent to the marked hexagons of Figure~\ref{fig:clustershex}, and to the marked Spectres of Figure~\ref{fig:markedSpectres}. The markings record the possible ways the unmarked versions of the clusters in Figure~\ref{fig:thespectretiles} can fit together in tilings. These clusters define supertiles for substitution rules on the marked Spectres, which may be coloured the same way as the combinatorially equivalent rules of Figure~\ref{fig:supertiles}.}
\label{fig:clusters}
\end{figure}
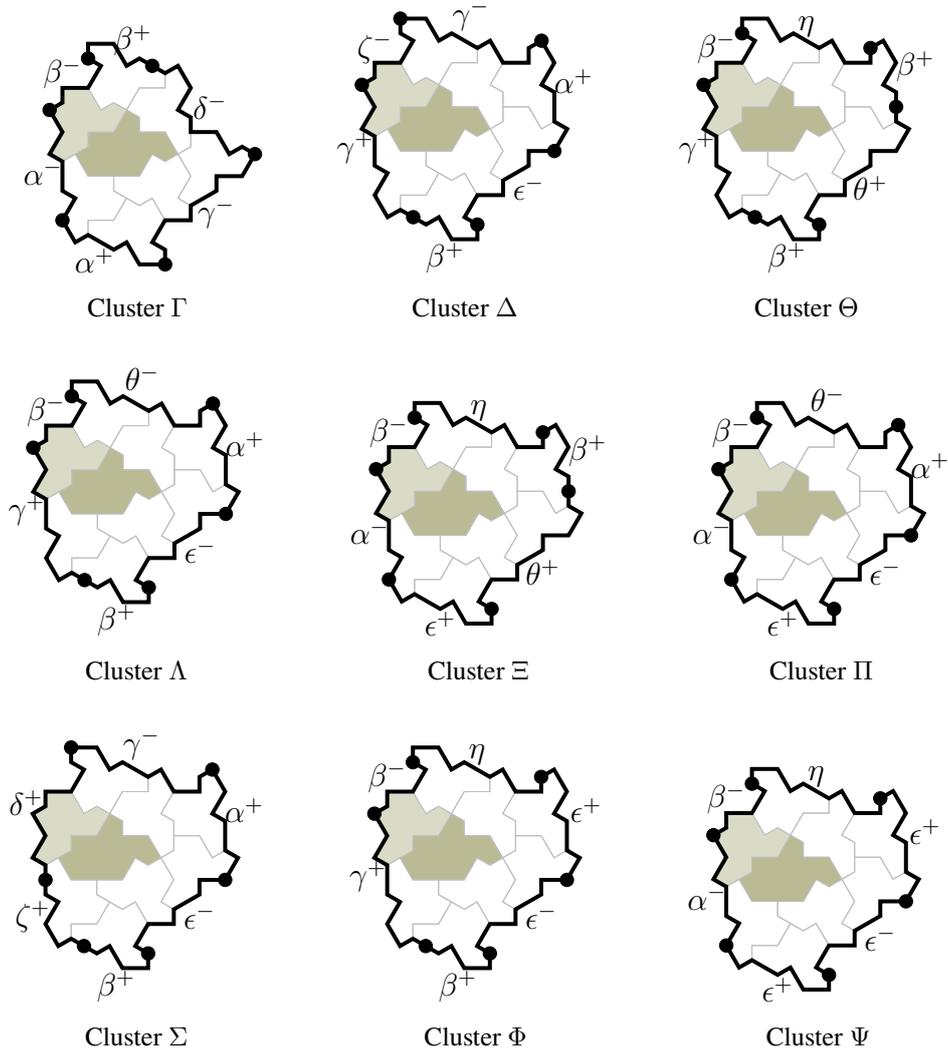


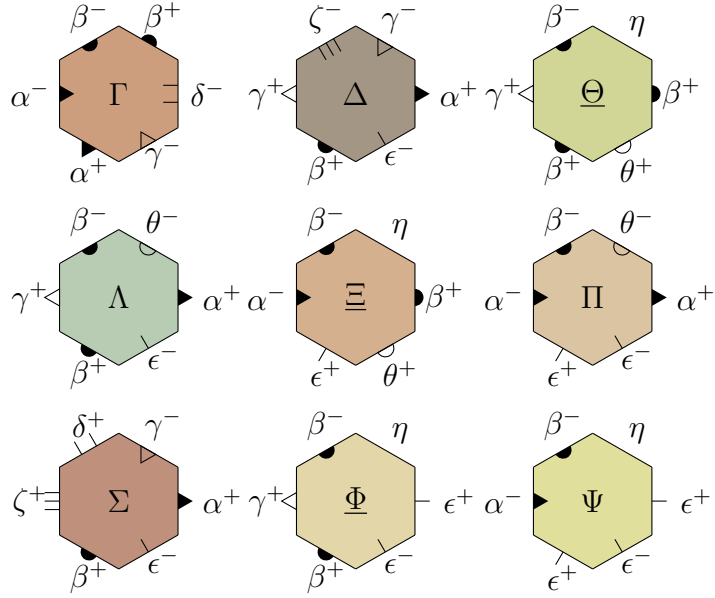
\begin{figure}[ht!]
\begin{center}
\begin{tikzpicture}[x=4.5mm,y=4.5mm]
 \hexGamma{0}{0}{0};
  \vctxt{-1}{1.5}{$\alpha^-$};
  \vctxt{1}{-1.75}{$\alpha^+$};
  \vctxt{3.5}{-2.5}{$\gamma^-$};
  \vctxt{5}{-1.5}{$\delta^-$};
  \vctxt{3.5}{1.5}{$\beta^+$};
  \vctxt{1}{2.75}{$\beta^-$};
  \hexDelta{0}{2}{0};
  \vctxt{7}{-2.5}{$\gamma^+$};
  \vctxt{9}{-5.75}{$\beta^+$};
  \vctxt{11.5}{-6.5}{$\epsilon^-$};
  \vctxt{13.5}{-5.75}{$\alpha^+$};
  \vctxt{11.5}{-2.5}{$\gamma^-$};
  \vctxt{9}{-1.25}{$\zeta^-$};
  \hexTheta{0}{4}{0};
  \vctxt{15}{-6.5}{$\gamma^+$};
  \vctxt{17}{-9.75}{$\beta^+$};
  \vctxt{19.5}{-11}{$\theta^+$};
  \vctxt{21}{-9.5}{$\beta^+$};
  \vctxt{19.5}{-6.75}{$\eta$};
  \vctxt{17}{-5.25}{$\beta^-$};
  \hexLambda{0}{1}{-2};
  \vctxt{-1}{-4.5}{$\gamma^+$};
  \vctxt{1}{-7.75}{$\beta^+$};
  \vctxt{3.5}{-8.5}{$\epsilon^-$};
  \vctxt{5.5}{-7.75}{$\alpha^+$};
  \vctxt{3.5}{-4.5}{$\theta^-$};
  \vctxt{1}{-3.25}{$\beta^-$};
  \hexXi{0}{3}{-2};
  \vctxt{7}{-8.5}{$\alpha^-$};
  \vctxt{9}{-11.75}{$\epsilon^+$};
  \vctxt{11.5}{-13}{$\theta^+$};
  \vctxt{13}{-11.5}{$\beta^+$};
  \vctxt{11.5}{-8.75}{$\eta$};
  \vctxt{9}{-7.25}{$\beta^-$};
  \hexPi{0}{5}{-2};
  \vctxt{15}{-12.5}{$\alpha^-$};
  \vctxt{17}{-15.75}{$\epsilon^+$};
  \vctxt{19.5}{-16.5}{$\epsilon^-$};
  \vctxt{21.5}{-15.75}{$\alpha^+$};
  \vctxt{19.5}{-12.5}{$\theta^-$};
  \vctxt{17}{-11.25}{$\beta^-$};
  \hexSigma{0}{2}{-4};
  \vctxt{-1}{-10.5}{$\zeta^+$};
  \vctxt{1}{-13.75}{$\beta^+$};
  \vctxt{3.5}{-14.5}{$\epsilon^-$};
  \vctxt{5.5}{-13.75}{$\alpha^+$};
  \vctxt{3.5}{-10.5}{$\gamma^-$};
  \vctxt{1}{-9.25}{$\delta^+$};
  \hexPhi{0}{4}{-4};
  \vctxt{7}{-14.5}{$\gamma^+$};
  \vctxt{9}{-17.75}{$\beta^+$};
  \vctxt{11.5}{-18.5}{$\epsilon^-$};
  \vctxt{13.5}{-17.75}{$\epsilon^+$};
  \vctxt{11.5}{-14.75}{$\eta$};
  \vctxt{9}{-13.25}{$\beta^-$};
  \hexPsi{0}{6}{-4};
  \vctxt{15}{-18.5}{$\alpha^-$};
  \vctxt{17}{-21.75}{$\epsilon^+$};
  \vctxt{19.5}{-22.5}{$\epsilon^-$};
  \vctxt{21.5}{-21.75}{$\epsilon^+$};
  \vctxt{19.5}{-18.75}{$\eta$};
  \vctxt{17}{-17.25}{$\beta^-$};
\end{tikzpicture}
\end{center}
\caption{Nine marked hexagons, combinatorially equivalent to the marked 
	clusters of Figure~\ref{fig:clusters}.  Hexagon edges are marked with
	the same lowercase Greek letters as the clusters; edges are also 
	decorated with inward- or outward-facing geometric markings that must
	agree for adjacent hexagons.}
\label{fig:clustershex}
\end{figure}

The first cluster, $\Gamma$, consists of eight tiles: one turtle
surrounded by seven hats.  The other eight clusters all consist of
the same arrangement of nine tiles, namely a $\Gamma$ cluster with
one additional hat.  These eight nine-tile clusters are distinguished based
on the locations marked for their tiling vertices in a tiling by
clusters.  
Edge segments with the same letter and opposite signs (or no sign) 
must adjoin each other on adjacent tiles.
The nine clusters taken together can be regarded abstractly
(and concretely, using combinatorial equivalence) as regular
hexagons with marked edges, as visualized in Figure~\ref{fig:clustershex}.

 \begin{proof}[Proof of Theorem~\ref{thm:clusters89}]
We first assume without loss
of generality that the tiling by Spectres and the corresponding tiling
by hats and turtles consist entirely of unreflected tiles. If all tiles
are reflected, we reflect the whole tiling, perform the analysis below,
and reflect the resulting clusters to match the original tiling.

Our proof is based on a computer-assisted case analysis similar to
the one used by Smith et~al.\ to establish the aperiodicity of the
hat~\cite[Appendix~B]{hat}.  This analysis depends on computing
``reduced lists of $1$-patches'' for sets of tiles, by first
generating a large list of candidate $1$-patches (defined below),
and then ``reducing'' the list by eliminating those patches that
cannot occur in tilings.

Given a set of tiles, define a \emph{$0$-patch} to be a single tile.
Let $\mathcal{P}$ be a $0$-patch and let 
$\mathcal{S}$ be a set of tiles whose interiors
are pairwise disjoint with each other and with $\mathcal{P}$.  If every tile
of $\mathcal{S}$ is a neighbour of $\mathcal{P}$, and if the union of 
$\mathcal{P}$ and the tiles
in $\mathcal{S}$ is a topological disk with $\mathcal{P}$ in its interior, 
then we call
$\mathcal{S}$ a \emph{$1$-corona} of $\mathcal{P}$, and we call 
$\mathcal{P}\cup \mathcal{S}$ a 
\emph{$1$-patch}.  More generally, any $k$-patch may be completely
surrounded by a $(k+1)$-corona to yield a $(k+1)$-patch.
We use ``corona'' as a shorthand for ``$1$-corona''.

To construct a reduced list of $1$-patches for a given a set of tiles,
first generate a list of legal pairs of
neighbouring tiles, eliminating neighbour relationships that can
be determined not to be extendable to a tiling~\cite[Appendix~B.1]{hat}.
If we wish to restrict our attention to homochiral tilings, we also
discard all pairs of tiles with opposite handedness.  Next, generate
a complete list of legal $1$-patches of tiles, comprising all ways of 
surrounding a single central tile by a corona of tiles, yielding a patch in
which all neighbours are legal as computed above.

Let~$\mathcal{L}$ be a list of $1$-patches and~$\mathcal{P}$ a patch
in~$\mathcal{L}$.  We say that a tile~$T$ in~$\mathcal{P}$ is
``$\mathcal{L}$-extendable'' if we can augment~$\mathcal{P}$ with
tiles to complete a corona around~$T$ that matches one of the patches
in~$\mathcal{L}$, in such a way that neighbour pairs in the augmented
patch are all legal.  (It may also be possible to surround~$T$ with
other coronas not in~$\mathcal{L}$, but any such coronas are not
relevant here.)  The $1$-patch~$\mathcal{P}$ is $\mathcal{L}$-extendable
if all the tiles in its corona are.  Now let~$\mathcal{L}_1$ be the
complete list of legal $1$-patches.  Any patch in~$\mathcal{L}_1$
that is not $\mathcal{L}_1$-extendable cannot appear in a tiling.
We remove any non-extendable patches to produce a smaller list
$\mathcal{L}_2$.  We iterate this process, producing progressively
smaller lists of patches, until there are no further removals.  The
resulting reduced list contains all $1$-patches that occur in full
tilings, and possibly a few false positives do not interfere with
our analysis below.  In practice, we test whether $T$ is
$\mathcal{L}$-extendable not by attempting to build coronas around
it, but by superimposing the central tile of every patch in~$\mathcal{L}$
on it and testing for compatibility of the two overlapping patches.

In the next section, we will also use the extended notion of reduced lists
of $k$-patches for~$k\!>\!1$.  The computation is analogous to that of
reduced lists of $1$-patches.  A list of $k$-patches may be generated
from a list of $(k-1)$-patches by considering ways to surround the
tiles in the inner $1$-corona of each $(k-1)$-patch with one of the
$(k-1)$-patches, and then the resulting list may be reduced by
considering whether each tile in a patch can be compatibly surrounded
by one of the patches in the list.

By computing the reduced $1$-patches of a set consisting of an unreflected
hat and an unreflected turtle, allowing only neighbours of the same
handedness, we obtain (a superset of) the $1$-patches that occur in
homochiral tilings by hats and turtles.  We find that every reduced 
$1$-patch with a turtle at its centre has at most one hat in its
corona, with the exception of a single $1$-patch where a turtle is
surrounded entirely by hats.  Likewise, every $1$-patch with a hat 
at centre has at most one turtle in its corona, except for a
single $1$-patch containing a hat surrounded by turtles.

By considering a connected sequence of tiles from any tile in a
tiling to any other, it follows that if any turtle has a turtle
neighbour, then every hat is entirely surrounded by turtles that
have no other hat neighbour, and if any hat has a hat neighbour,
then every turtle is entirely surrounded by hats that have no other
turtle neighbour.  This behaviour can be seen in the patches of
Figure~\ref{figure:hatturtlespectre}, where turtles are distributed
sparsely among hats or vice versa. 
Every tiling by the Spectre is equivalent to one tiling of each of
these two types, 
depending on which edges have their lengths changed from $1$
to~$\sqrt{3}$.
Without loss of generality we
will work with sparse turtles embedded in a field of hats.

Henceforth we refer to the unique $1$-patch consisting of a turtle
surrounded by hats, shown in Figure~\ref{fig:turtle678hats}
(left), as T6H (``turtle and six hats''). 
Because this patch is unique in the reduced list of 
$1$-patches, and because copies of the patch cannot overlap (as a
hat has at most one turtle neighbour), homochiral
tilings by hats and turtles correspond to homochiral tilings by two
unreflected tiles: hats and copies of T6H.

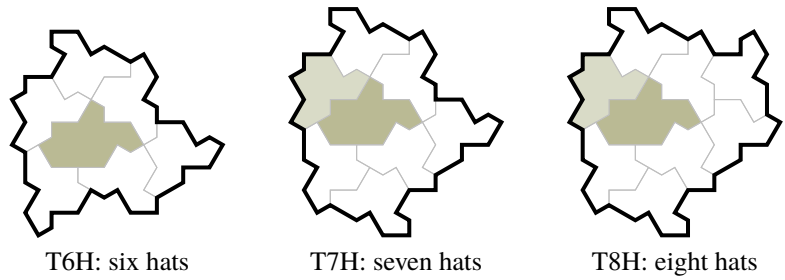
\begin{figure}[t]
\captionsetup{margin=0pt}%
\begin{center}
\subfloat[T6H: six hats]{%
\begin{tikzpicture}[x=2mm,y=2mm]
  \hexsevenB{};
\end{tikzpicture}%
} \qquad \subfloat[T7H: seven hats]{%
\begin{tikzpicture}[x=2mm,y=2mm]
  \hexeightB{};
\end{tikzpicture}%
} \qquad \subfloat[T8H: eight hats]{%
\begin{tikzpicture}[x=2mm,y=2mm]
  \hexnineB{};
\end{tikzpicture}%
}%
\end{center}
\caption{Clusters containing a turtle and six, seven or eight hats.}
\label{fig:turtle678hats}
\end{figure}

We now extend this analysis by constructing a reduced list of
$1$-patches from the set containing the hat and T6H.  In this list,
we discover that every $1$-patch with T6H at its centre must have
a hat neighbour in the appropriate position to form the cluster
T7H, consisting of a turtle and seven hats, shown in
Figure~\ref{fig:turtle678hats} (centre).  In a homochiral tilings by
hats and turtles, we already know that the copies of T6H identified
above cannot overlap.  So if two copies of T7H overlap, they must do
so by sharing a hat.\footnote{At this point, the tile T6H is
considered as an indivisible unit, not a cluster made up of hats and a
turtle.  So each T7H contains only a single hat that might potentially
overlap with a hat in another T7H.}  But the only way for them to share
a hat would be for the two copies of T7H to be identical.  Thus, copies
of T7H cannot
overlap either.  So we continue the grouping process, constructing
a reduced list of $1$-patches from the hat and T7H.  In this list,
every $1$-patch with the hat at its centre has a T7H neighbour in
the appropriate position to form the cluster T8H shown in
Figure~\ref{fig:turtle678hats} (right).  Copies of T7H and T8H
cannot overlap in a tiling by hats and turtles.  Finally, then,
in every homochiral tiling by hats and turtles the tiles can be composed
into copies of T7H and T8H.  As no arbitrary choices were made
in forming these clusters, the tiling has  precisely the same  symmetries as
the original tiling by hats and turtles. 

Next consider a reduced list of $1$-patches by T7H and T8H.  Classify
the central tile in such a $1$-patch according to the points on its
boundary that it shares with two of its neighbours (i.e., the patch's
tiling vertices).  This classification produces the nine marked clusters
listed in Figure~\ref{fig:clusters}.  The edge labels follow
immediately from the shapes of the boundary segments between the marked
vertices.  Each of these tiles has exactly six vertices of degree~$3$,
the central
tile in each $1$-patch has exactly six neighbours,
and the relation between the orientation of a
boundary segment and that of the corresponding edge in
Figure~\ref{fig:clustershex} is consistent.  Together, these observations
ensure that tilings by T7H and T8H are in bijection with combinatorially
equivalent tilings by the marked hexagons of Figure~\ref{fig:clustershex}.
\end{proof}

\FloatBarrier

\section{A substitution system for marked hexagons}
\label{sec:hexsubst}

In this section, we prove that any tiling admitted by the marked
hexagons of Figure~\ref{fig:clustershex} can be uniquely composed
into the supertiles of Figure~\ref{fig:supertiles}.  This composition
yields a unique hierarchy of level-$n$ supertiles for all~$n$, which 
forces any tiling by these
hexagons to be non-periodic.  We also observe that the supertiles
imply a substitution system that produces patches of hexagons of
any size, thus confirming that the hexagons tile the plane.

\begin{theorem}
\label{thm:subst}
The tiles in any tiling by the nine marked hexagons of
Figure~\ref{fig:clustershex} can be composed into the supertiles shown
in Figure~\ref{fig:supertiles}. The supertiles are combinatorially 
equivalent to reflected versions of the hexagons.  The tilings by supertiles 
and marked hexagons have the same symmetries.
\end{theorem}

\begin{figure}[htp!]
\captionsetup{margin=0pt}%
\begin{center}
\subfloat[Supertile $\Gamma$]{%
\begin{tikzpicture}[x=3.75mm,y=3.75mm]
  \hexGamma{0}{0}{0};
  \hexSigma{60}{1}{0};
  \hexPhi{120}{0}{1};
  \hexXi{60}{1}{1};
  \hexPi{240}{0}{-1};
  \hexDelta{300}{1}{-1};
  \hexTheta{0}{2}{-1};
  \markpt{4}{4};
  \markpt{0}{2};
  \markpt{-2}{0};
  \markpt{0}{-4};
  \markpt{10}{-8};
  \markpt{10}{0};
  \vctxt{1}{3}{$\alpha^-$};
  \vctxt{-1}{1}{$\beta^-$};
  \vctxt{-3}{-2}{$\beta^+$};
  \vctxt{5}{-7}{$\delta^-$};
  \vctxt{9}{-3.5}{$\gamma^-$};
  \vctxt{7}{2.75}{$\alpha^+$};
\end{tikzpicture}%
} \qquad \subfloat[Supertile $\Delta$]{%
\begin{tikzpicture}[x=3.75mm,y=3.75mm]
  \hexGamma{0}{0}{0};
  \hexSigma{60}{1}{0};
  \hexPhi{120}{0}{1};
  \hexPi{60}{1}{1};
  \hexXi{240}{0}{-1};
  \hexDelta{300}{1}{-1};
  \hexPhi{0}{2}{-1};
  \hexXi{300}{2}{-2};
  \markpt{8}{2};
  \markpt{0}{2};
  \markpt{-2}{-2};
  \markpt{6}{-10};
  \markpt{10}{-8};
  \markpt{10}{-2};
  \vctxt{3}{5}{$\gamma^+$};
  \vctxt{-1}{1}{$\zeta^-$};
  \vctxt{2}{-5}{$\gamma^-$};
  \vctxt{9}{-9}{$\alpha^+$};
  \vctxt{9}{-3.5}{$\epsilon^-$};
  \vctxt{10}{1}{$\beta^+$};
\end{tikzpicture}%
} \qquad \subfloat[Supertile $\Theta$]{%
\begin{tikzpicture}[x=3.75mm,y=3.75mm]
  \hexGamma{0}{0}{0};
  \hexSigma{60}{1}{0};
  \hexPhi{120}{0}{1};
  \hexPi{60}{1}{1};
  \hexPsi{240}{0}{-1};
  \hexDelta{300}{1}{-1};
  \hexPhi{0}{2}{-1};
  \hexPi{300}{2}{-2};
  \markpt{8}{2};
  \markpt{0}{2};
  \markpt{-2}{0};
  \markpt{4}{-8};
  \markpt{8}{-10};
  \markpt{10}{-2};
  \vctxt{3}{5}{$\gamma^+$};
  \vctxt{-1}{1}{$\beta^-$};
  \vctxt{2}{-5}{$\eta$};
  \vctxt{6}{-11}{$\beta^+$};
  \vctxt{9}{-3.5}{$\theta^+$};
  \vctxt{10}{1}{$\beta^+$};
\end{tikzpicture}%
} \\ \subfloat[Supertile $\Lambda$]{%
\begin{tikzpicture}[x=3.75mm,y=3.75mm]
  \hexGamma{0}{0}{0};
  \hexSigma{60}{1}{0};
  \hexPhi{120}{0}{1};
  \hexPi{60}{1}{1};
  \hexPsi{240}{0}{-1};
  \hexDelta{300}{1}{-1};
  \hexPhi{0}{2}{-1};
  \hexXi{300}{2}{-2};
  \markpt{8}{2};
  \markpt{0}{2};
  \markpt{-2}{0};
  \markpt{6}{-10};
  \markpt{10}{-8};
  \markpt{10}{-2};
  \vctxt{3}{5}{$\gamma^+$};
  \vctxt{-1}{1}{$\beta^-$};
  \vctxt{2}{-5}{$\theta^-$};
  \vctxt{9}{-9}{$\alpha^+$};
  \vctxt{9}{-3.5}{$\epsilon^-$};
  \vctxt{10}{1}{$\beta^+$};
\end{tikzpicture}%
} \qquad \subfloat[Supertile $\Xi$]{%
\begin{tikzpicture}[x=3.75mm,y=3.75mm]
  \hexGamma{0}{0}{0};
  \hexSigma{60}{1}{0};
  \hexPhi{120}{0}{1};
  \hexPsi{60}{1}{1};
  \hexPsi{240}{0}{-1};
  \hexDelta{300}{1}{-1};
  \hexPhi{0}{2}{-1};
  \hexPi{300}{2}{-2};
  \markpt{4}{4};
  \markpt{0}{2};
  \markpt{-2}{0};
  \markpt{4}{-8};
  \markpt{8}{-10};
  \markpt{10}{-2};
  \vctxt{1}{3}{$\alpha^-$};
  \vctxt{-1}{1}{$\beta^-$};
  \vctxt{2}{-5}{$\eta$};
  \vctxt{6}{-11}{$\beta^+$};
  \vctxt{9}{-3.5}{$\theta^+$};
  \vctxt{6}{3}{$\epsilon^+$};
\end{tikzpicture}%
} \qquad \subfloat[Supertile $\Pi$]{%
\begin{tikzpicture}[x=3.75mm,y=3.75mm]
  \hexGamma{0}{0}{0};
  \hexSigma{60}{1}{0};
  \hexPhi{120}{0}{1};
  \hexPsi{60}{1}{1};
  \hexPsi{240}{0}{-1};
  \hexDelta{300}{1}{-1};
  \hexPhi{0}{2}{-1};
  \hexXi{300}{2}{-2};
  \markpt{4}{4};
  \markpt{0}{2};
  \markpt{-2}{0};
  \markpt{6}{-10};
  \markpt{10}{-8};
  \markpt{10}{-2};
  \vctxt{1}{3}{$\alpha^-$};
  \vctxt{-1}{1}{$\beta^-$};
  \vctxt{2}{-5}{$\theta^-$};
  \vctxt{9}{-9}{$\alpha^+$};
  \vctxt{9}{-3.5}{$\epsilon^-$};
  \vctxt{6}{3}{$\epsilon^+$};
\end{tikzpicture}%
} \\ \subfloat[Supertile $\Sigma$]{%
\begin{tikzpicture}[x=3.75mm,y=3.75mm]
  \hexGamma{0}{0}{0};
  \hexSigma{60}{1}{0};
  \hexLambda{120}{0}{1};
  \hexPi{60}{1}{1};
  \hexXi{240}{0}{-1};
  \hexDelta{300}{1}{-1};
  \hexPhi{0}{2}{-1};
  \hexXi{300}{2}{-2};
  \markpt{8}{2};
  \markpt{2}{4};
  \markpt{-2}{-2};
  \markpt{6}{-10};
  \markpt{10}{-8};
  \markpt{10}{-2};
  \vctxt{6}{3}{$\zeta^+$};
  \vctxt{-1}{1}{$\delta^+$};
  \vctxt{2}{-5}{$\gamma^-$};
  \vctxt{9}{-9}{$\alpha^+$};
  \vctxt{9}{-3.5}{$\epsilon^-$};
  \vctxt{10}{1}{$\beta^+$};
\end{tikzpicture}%
} \qquad \subfloat[Supertile $\Phi$]{%
\begin{tikzpicture}[x=3.75mm,y=3.75mm]
  \hexGamma{0}{0}{0};
  \hexSigma{60}{1}{0};
  \hexPhi{120}{0}{1};
  \hexPi{60}{1}{1};
  \hexPsi{240}{0}{-1};
  \hexDelta{300}{1}{-1};
  \hexPhi{0}{2}{-1};
  \hexPsi{300}{2}{-2};
  \markpt{8}{2};
  \markpt{0}{2};
  \markpt{-2}{0};
  \markpt{4}{-8};
  \markpt{10}{-8};
  \markpt{10}{-2};
  \vctxt{3}{5}{$\gamma^+$};
  \vctxt{-1}{1}{$\beta^-$};
  \vctxt{2}{-5}{$\eta$};
  \vctxt{9}{-9}{$\epsilon^+$};
  \vctxt{9}{-3.5}{$\epsilon^-$};
  \vctxt{10}{1}{$\beta^+$};
\end{tikzpicture}%
} \qquad \subfloat[Supertile $\Psi$]{%
\begin{tikzpicture}[x=3.75mm,y=3.75mm]
  \hexGamma{0}{0}{0};
  \hexSigma{60}{1}{0};
  \hexPhi{120}{0}{1};
  \hexPsi{60}{1}{1};
  \hexPsi{240}{0}{-1};
  \hexDelta{300}{1}{-1};
  \hexPhi{0}{2}{-1};
  \hexPsi{300}{2}{-2};
  \markpt{4}{4};
  \markpt{0}{2};
  \markpt{-2}{0};
  \markpt{4}{-8};
  \markpt{10}{-8};
  \markpt{10}{-2};
  \vctxt{1}{3}{$\alpha^-$};
  \vctxt{-1}{1}{$\beta^-$};
  \vctxt{2}{-5}{$\eta$};
  \vctxt{9}{-9}{$\epsilon^+$};
  \vctxt{9}{-3.5}{$\epsilon^-$};
  \vctxt{6}{3}{$\epsilon^+$};
\end{tikzpicture}%
}%
\end{center}
\caption{Nine supertiles.  Each is drawn in the reverse handedness of the corresponding marked hexagon of Figure~\ref{fig:clustershex}, preserving the handedness of the marked hexagons within it.}
\label{fig:supertiles}
\end{figure}
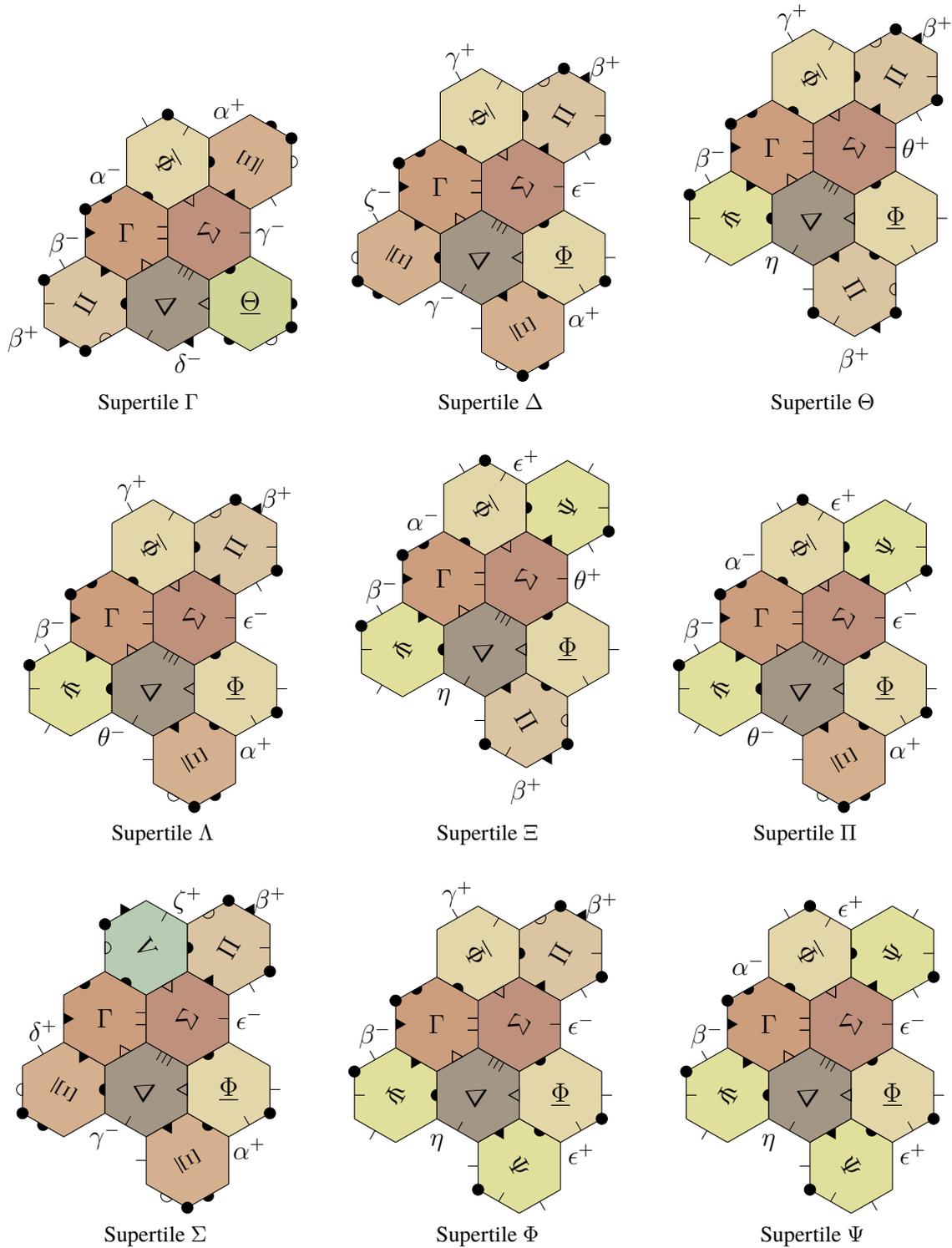

\begin{proof}
Each supertile in Figure~\ref{fig:supertiles}
contains a $\Gamma$~hexagon and either six or seven other hexagons:
Supertile~$\Gamma$ contains six other hexagons, while all the other
supertiles contain hexagons in those six positions plus a seventh
hexagon.  Each supertile is annotated with the positions of its
vertices and  markings on the edges between those vertices.  By
inspection,  the edge markings indicate portions of the supertile
boundaries that match if and only if the corresponding marked edges
of the original marked hexagons match.  Furthermore, exactly three
supertiles must meet at each vertex. Thus the supertiles are indeed
combinatorially equivalent to reflected versions of the marked
hexagons.

To prove Theorem~\ref{thm:subst}, it suffices then to give rules to
compose the marked hexagons of any tiling into the nine supertiles
shown, such that vertices are in the positions marked, with no
arbitrary choices involved in those rules (so that the tiling by
supertiles necessarily has the same symmetries as the original tiling
by marked hexagons).  The lack of a translation as a symmetry then
follows from the sizes of supertiles after $n$~composition steps going
to infinity with~$n$, while the substitution structure implies that
the marked hexagons tile arbitrarily large finite regions of the
plane, and thus the whole plane.

We accrete one supertile around each $\Gamma$~hexagon.  First assign to that
supertile the $\Gamma$~hexagon itself and the six other hexagons in
the positions relative to that~$\Gamma$ that appear in all of the
illustrated supertiles.  By constructing a reduced list of $5$-patches by
the marked hexagons, we can confirm that this assignment associates each marked 
hexagon in any tiling with at most one supertile, but leaves some hexagons
unassigned.  In the eight supertiles other than
Supertile~$\Gamma$, the remaining hexagon is in a fixed position
relative to the~$\Gamma$ hexagon.  For each unassigned hexagon, assign
it to the supertile where it is in the right position relative to that
supertile's $\Gamma$~hexagon; again, examining a reduced list of
$5$-patches confirms that every unassigned hexagon is assigned by
this rule to exactly one supertile.

It remains to show that the supertiles resulting from these assignment
rules are exactly those shown in Figure~\ref{fig:supertiles} (i.e.,
that they contain marked hexagons in the positions and orientations
illustrated, and have vertices in the indicated positions).  For
this step, consider those reduced $5$-patches with $\Gamma$~hexagons at
their centres.  Examining the list of such patches shows that there
are exactly the nine possibilities illustrated
for the labels and orientations of all the tiles
in the supertile of that~$\Gamma$.  Furthermore, for each such
possibility, there is a unique division of the hexagons surrounding
that supertile into their own supertiles, which determines the
locations of the vertices of the supertile (this is the step that
requires $5$-patches, to assign those neighbouring hexagons to their
own supertiles, rather than working with $k$-patches for some $k <
5$).  That unique choice of the locations of the vertices is the one
shown in Figure~\ref{fig:supertiles}.
\end{proof}

\begin{figure}[t]
\centerline{\includegraphics[width= \textwidth]{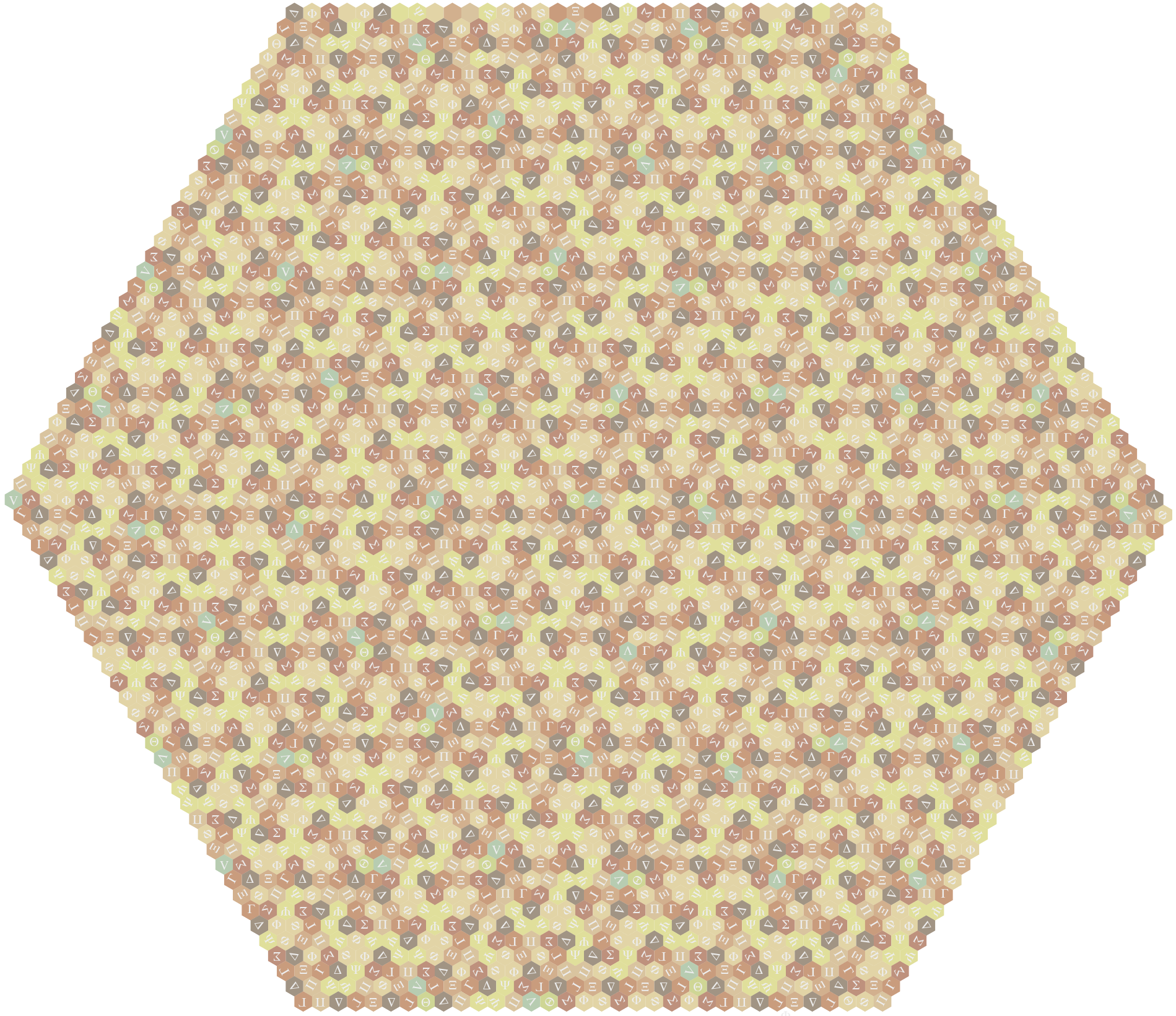}
}
\caption{\label{fig:bighexpatch}A large patch of marked hexagons.}\end{figure}

Theorem~\ref{thm:subst} allows us to compose any tiling by marked
hexagons into a combinatorially equivalent tiling by supertiles.
By induction, we can then iterate this composition process to any
number of steps, gathering level-$1$ supertiles into level-$2$
super-supertiles, and so on, yielding level-$n$ supertiles for all~$n$.
Every marked hexagon in a tiling is therefore
nested in a unique way within an infinite hierarchy of level-$n$
supertiles.  It follows that any tiling
by the marked hexagons is non-periodic, which implies the non-periodicity
first of tilings by copies of~T7H and~T8H, then of homochiral tilings by 
hats and turtles, and finally of tilings
by Spectres, all through combinatorial equivalence.  

To complete a proof of aperiodicity, we must also show that any of
these families of shapes actually admits tilings.  Here we need
only ``reverse'' the composition process to obtain substitution
rules that replace each marked hexagon in Figure~\ref{fig:clustershex} 
by the corresponding supertile
in Figure~\ref{fig:supertiles}.  These rules are well-defined,
in the sense that they may be iterated to produce level-$n$ supertiles
for every $n\ge 0$, where the supertiles are considered as combinatorial
patches (collections of marked hexagons with associated adjacency
relations) rather than geometric ones (marked hexagons in particular
positions and orientations in the plane).  We can therefore construct
combinatorial patches of marked hexagons of any size; Figure~\ref{fig:bighexpatch} gives an example of a large patch.
From any such patch we can derive a combinatorial patch of Spectres,
which in the limit implies the existence of an unbounded combinatorial
patch.  Because of the consistency of edge labels and shapes of boundary
segments in Figure~\ref{fig:clusters}, and the consistency of the
relation between the orientation of a boundary segment and that of the
corresponding edge (as noted at the end of Section~\ref{sec:hexagons}),
every combinatorial edge and vertex arrangement that arises in that
patch can be realized geometrically, and so we may apply the geometry
of the marked hexagons, or the geometry of the Spectre clusters,
to any combinatorial patch to construct a geometric tiling of the
plane~\cite[Lemma~1.1]{regprod}.  Thus the Spectre tiles the plane,
which completes the proof of Theorem~\ref{maintheorem}: the Spectre
is a strictly chiral aperiodic monotile.

\begin{figure}[t]
\centerline{\includegraphics[width=.7\textwidth, trim = 0in .05in 0in .05in]{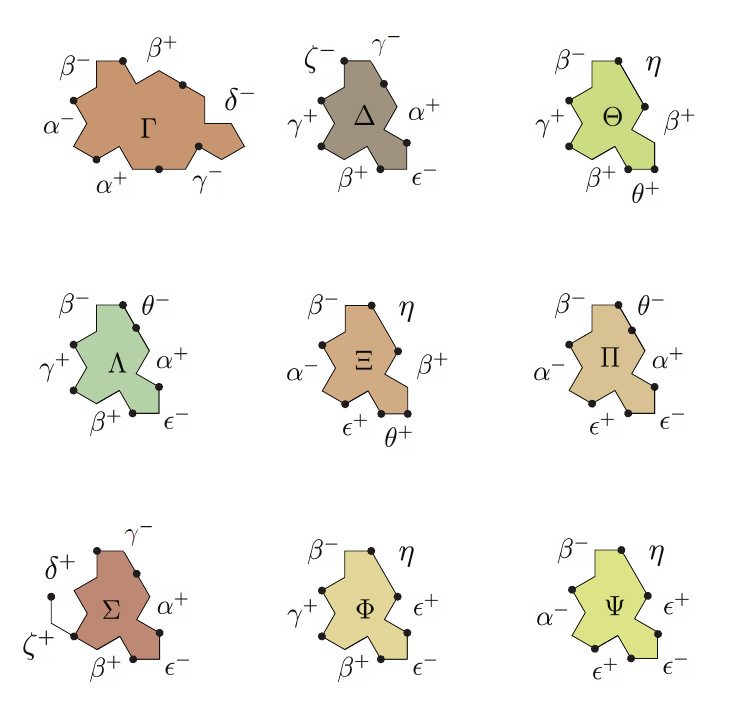}} 
\caption{
Marked Spectres that are combinatorially equivalent to the marked
hexagons of Figure~\ref{fig:clustershex}  and to the
clusters of Figure~\ref{fig:clusters}. The $\Sigma$~tile has
additional line segments drawn outside the tile, which should be
interpreted as the boundary of part of the tile with zero thickness;
the right-hand side of those segments forms part of the
$\delta^+$~edge, while the left-hand side is the $\zeta^+$~edge 
(similar segments arise in boundaries of the clusters forming
metatiles for the hat~\cite[Section~4]{hat}).
The labelled clusters define substitution rules on the marked
Spectres that are combinatorially equivalent to the substitution
rules on the marked hexagons, shown in Figure~\ref{fig:supertiles}.
}
\label{fig:markedSpectres}
\end{figure}

Most of our aperiodicity proof for the Spectre relies not on the
Spectres themselves, but on the marked clusters defined in
Figure~\ref{fig:clusters}.  It is convenient to position the proof
at this level: the clusters have the comparatively simple combinatorics
of regular hexagons, and our labelled edges make it easier to verify
that supertiles and clusters have combinatorially equivalent matching
rules.  However, it is also possible to take a step backwards
and annotate the Spectres themselves with these same markings.
Figure~\ref{fig:markedSpectres} shows a Mystic and eight Spectres,
with the same tile and edge markings as the associated marked
hexagons.  The tiling substitution rules on marked hexagons apply
equally to these marked Spectres.  Erasing the markings restores
the Spectre tiling substitution rules of Figure~\ref{fig:thespectretiles}.
Thus the Spectre enforces those substitution rules directly.



\section{Conclusion}
\label{sec:conclusion}

Many problems in tiling theory depend implicitly or explicitly on
an initial decision of when to consider two tiles in a tiling ``the
same''.  In the Euclidean plane, if this decision is not otherwise
articulated then we assume two tiles are the same if they can be
brought into coincidence through any planar isometry.  In this paper
we first answer the einstein problem in a world where sameness is
restricted to orientation-preserving isometries.  The polygon
$\mathrm{Tile}(1,1)$ is a \emph{weakly chiral aperiodic monotile}:
a shape that tiles aperiodically if only translations and rotations
are permitted.  Then, by modifying the edges of that polygon, we
obtain a class of shapes called Spectres, which are \emph{strictly
chiral aperiodic monotiles}: they tile aperiodically using tiles
of a single handedness, even when reflections are allowed.

Other variations of the einstein problem can be posed in which tiles
are restricted to images under any given group of isometries.  Each such
problem comes in weak and strict forms, as above.
In the weak case, we ask whether a shape's tilings are
non-periodic if tilings are restricted by fiat to the given isometries,
even if it would tile periodically when all isometries are allowed.  In
the strict case we require that by design, a shape does
not admit any tilings that use isometries outside the group.

If we restrict our attention to translations alone, then 
the work of Girault-Beauquier
and Nivat~\cite{GiraultBeauquierNivat} and
Kenyon~\cite{Kenyon,Kenyonerratum,Kenyon2} shows that a topological
disk with a tiling by translation cannot be an aperiodic monotile.
However, Greenfeld and Tao~\cite{GT1,GT2} and Greenfeld and
Kolountzakis~\cite{greenfeld2023tiling} showed that translational
aperiodic monotiles do exist in sufficiently high dimensions.
The next simplest case is where $180^\circ$~rotations are allowed
along with translations.
Schattschneider~\cite[Problem~18.E1]{shapingspace} posed the question
of whether any tile admitting such a tiling of the plane also
satisfies the Conway
criterion~\cite{gardnerconwaycriterion,conwaycriterion}, implying
that it must be isohedral.

Our work presents a single family of strictly chiral aperiodic
monotiles, which are all essentially the same up to trivial
modifications of tiling edges.  It would be interesting to search
for other weakly chiral or strictly chiral einsteins.
It would be particularly worthwhile to find (or disprove the existence
of) an aperiodic monotile with bilateral reflection symmetry, a
shape for which chirality becomes moot.

We showed in Lemma~\ref{existlemma} that we can replace the~$14$
edges of $\mathrm{Tile}(1,1)$ by suitably oriented copies of any
smooth path to construct a Spectre, provided the replacement results
in a tile with a non-self-intersecting boundary.  In fact, the
construction works generally for $C^1$ paths, which suffice to
force the vertices of the original polygon to meet each other in
tilings.  However, we leave open the question of whether some other
paths (e.g., piecewise-linear paths)  might permit different
tilings, or whether all non-trivial choices of path produce valid Spectres. 

As part of our proof of aperiodicity, we derived the chiral
marked hexagons of Figure~\ref{fig:clustershex}.  These hexagons
are meant to encode the combinatorics of clusters of Spectres,
but they display interesting properties of their own that may be
worthy of further study.  Figure~\ref{fig:bighexpatch} hints at
emergent patterns in tilings by marked hexagons. 

Some of the hexagons of Figure~\ref{fig:clustershex} have very
similar arrangements of markings; removing the distinctions between
some of the edge labels can produce a smaller set of marked hexagons
without reducing the set to a single trivially marked tile.  In
particular, combining $\Theta$, $\Xi$, $\Phi$ and~$\Psi$ into a
single marked hexagon, and combining $\Lambda$ and~$\Pi$ into another
marked hexagon, yields a smaller set of five marked hexagons, which
preliminary computations suggest might also be aperiodic.  It would
be of interest to understand what small aperiodic sets of chiral
marked hexagons are possible with this style of edge markings,
similar to the work of Jeandel and Rao~\cite{JeandelRao} on small
aperiodic sets of Wang tiles.



\section*{Acknowledgements}
\label{sec:acks}

Special thanks to Yoshiaki Araki, whose artistic explorations with the
equilateral (or nearly equilateral) member of the hat-turtle continuum
first inspired us to consider the tiling-theoretic properties of the
Spectre in detail.  Thanks also to Lucy Birkinshaw for pointing out an
error in one of the figures, to Jeff Wilson, Dave Keenan, and Doug
Blumeyer for suggesting improvements to the terminology used in the article.

\appendix

\section{Constructing Spectre tilings}
\label{spectresubstAppendix}

The methods used in this paper to prove the aperiodicity of
$\mathrm{Tile}(1,1)$ and Spectres rely heavily on layers of
combinatorial equivalence.  Because these methods do not require
geometrically rigid information about the locations of tiles, we
emerge at the end without a practical algorithm for drawing patches
of tiles. Such an algorithm would of course be useful for
visualization and artistic experimentation, and so in this appendix
we provide one.  This algorithm also forms the basis for an interactive
browser-based tool for constructing patches of Spectres, available
at~\href{https://cs.uwaterloo.ca/~csk/spectre/}{\nolinkurl{cs.uwaterloo.ca/~csk/spectre/}}.

In Section~\ref{sec:hexagons} we show, through a computer-assisted
analysis of cases, that any homochiral tiling by hats and turtles
must consist of a sparse arrangement of turtles surrounded by hats, or
sparse hats surrounded by turtles.
We then show that the sparse turtle tiling can be composed into copies
of two clusters, named T7H (``turtle and seven hats'') and T8H 
(``turtle and eight hats''), illustrated in Figure~\ref{fig:turtle678hats}.
Working backwards through the equivalence described in 
Section~\ref{sec:hatturtle}, from T7H and T8H we can recover corresponding
clusters of eight and nine Spectres.  
Figure~\ref{fig:thespectretiles} introduced these clusters,
rotated so that each contains a single ``odd'' Spectre
(shaded in dark green) surrounded by ``even'' Spectres.
The odd Spectres play a role similar to that played by reflected
hats in the hat tiling (except of course that they are not reflected):
they form a sparse subset, and each one is surrounded by a congruent
arrangement of neighbours.

These two clusters also form the basis for the substitution rules
shown in Figure~\ref{fig:thespectretiles}.  A corollary of the main
line of the paper is that these substitution rules can indeed 
produce patches of Spectres of any size.  We can also use these
rules to develop a practical drawing algorithm.  
As shown in Figure~\ref{fig:thespectretiles}, we combine one even
and one odd Spectre into a compound called a ``Mystic'', and define
substitution rules that replace Spectres and Mystics with clusters
of reflected tiles.

These rules are combinatorial---although they show the cluster that
replaces each Spectre or Mystic in a growing patch, they say nothing
about the exact location of that cluster.  To make the rules 
precise, we identify four ``key points'' on the boundary of the Spectre.  
We then use those points to determine the translations of the Spectres and
Mystics in substituted clusters, as well as next-generation key points on the
boundaries of those clusters.  These new key points are equivalent to the
originals, in that they can drive the same process of snapping clusters 
together into superclusters and the choice of their key points, and so on
through any number of generations.

We determined the locations of these key points through manual
experimentation.  We chose points on the boundaries of T7H and T8H
where three or more clusters meet, in such a way that every cluster
in a supercluster is connected to every other through a path of
neighbouring clusters with coincident key points (thus fully
determining the necessary translations mentioned above).  By comparing
the arrangement of Spectres in a cluster with the arrangement of
clusters in a supercluster, we could then deduce the corresponding
locations of key points on individual tiles.  In hindsight, the key
points used here are subsets of the marked dots in
Figures~\ref{fig:clusters} and~\ref{fig:supertiles}.  They are minimal
in the sense that we could have begun by distinguishing just two of the
marked dots as key points in each of the nine marked Spectres of
\fig{fig:markedSpectres}, and propagated those pairs of dots to the
nine cluster types and their superclusters.  Each of those pairs consists
of 
two of the four key points used here; operating with all four key points
at every stage avoids the complexity of tracking
nine cluster types.

\begin{figure}[t]
\begin{center}
	\includegraphics[width=\textwidth]{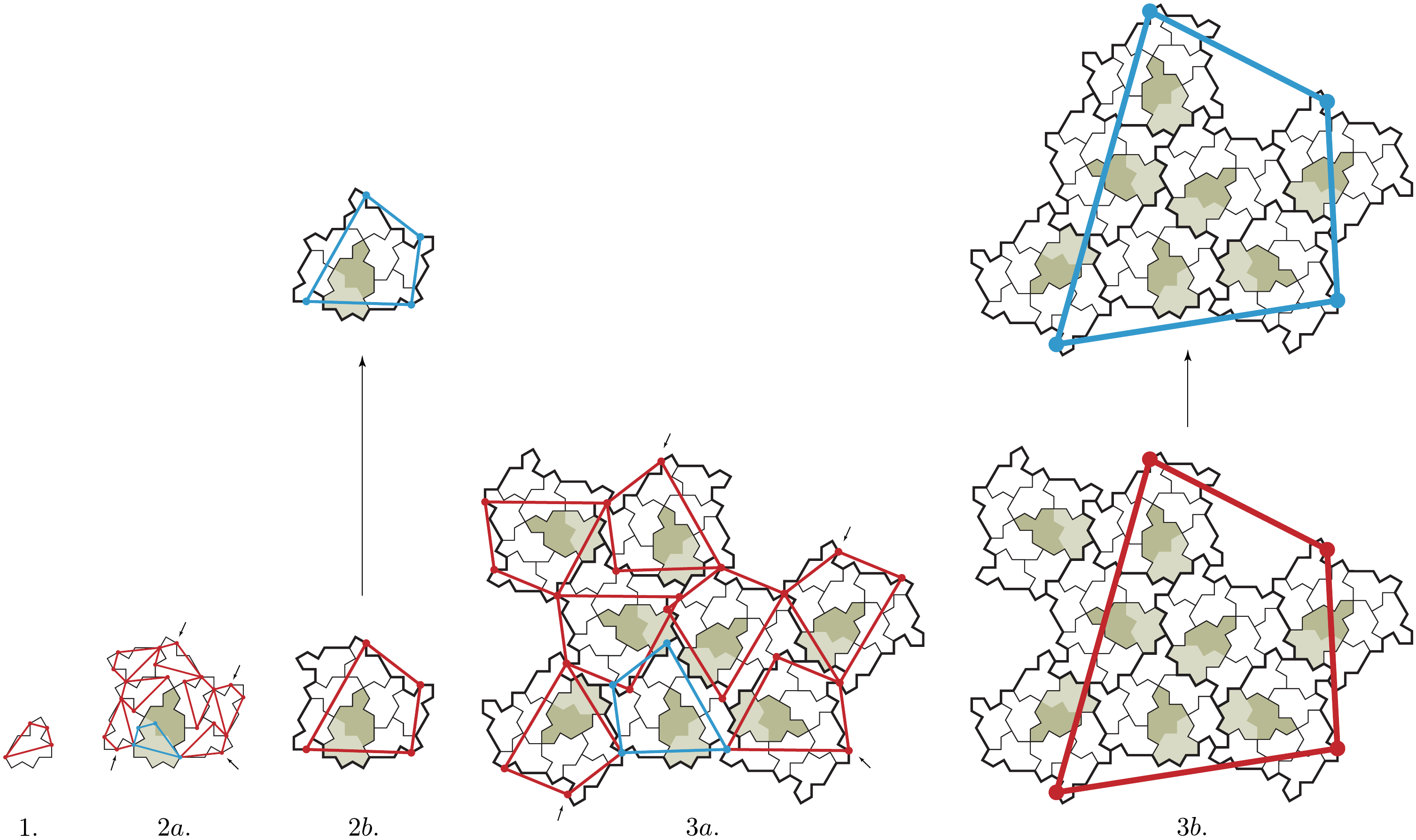}
\end{center}
\caption{\label{fig:patch_construction}
	A visualization of the algorithm for iterating Spectre and Mystic
	cluster construction.  A single Spectre is marked with four 
	``key points'' (Step~1), which guide the placement of Spectres and
	Mystics in the construction of clusters (Step~2a), as well as the
	choice of key points for those clusters (Step~2b).  The process
	may then be repeated to build superclusters (Steps~3a and~3b),
	super-superclusters, and so on.   Spectre clusters and Mystic clusters
	are marked with red points and blue points, respectively.}
\end{figure}

The construction is illustrated in Figure~\ref{fig:patch_construction}.
We begin by marking four
of the Spectre's vertices as key points in Step~1.  In the figure we 
draw the quadrilateral joining the points for visualization purposes
only; it is not explicitly needed in the construction.
We can then build a Spectre cluster as shown in Step~2a.  Note that the
sets of key points associated with the Mystic and Spectres form a closed
chain: each tile shares exactly two key points with neighbouring tiles.  We
can therefore construct the cluster by placing one Spectre at any desired
location and orientation as a starting point, and then gluing on neighbours 
by following the chain.
At each step we rotate the tile to be placed so it has the correct 
orientation relative to its predecessor as shown in Step~2a, and compute
the translation needed to bring the associated key points of the
neighbouring tiles into coincidence.
In Step~2a, four key points belonging to different Spectre tiles are indicated
with arrows; these become the key points of the Spectre cluster itself in 
Step~2b.  Note that we can also easily define the Mystic cluster by 
deleting a single even Spectre, as shown above the Spectre cluster.

We can then repeat this process, assembling Spectre and Mystic
clusters into a Spectre supercluster (Step~3a), and extracting the
supercluster's key points (Step~3b) as a subset of the key points
of the clusters that make it up, as was done to compute the key points
of those clusters.  By deleting one of the placed clusters, we also
construct the Mystic supercluster.  This construction can be iterated
as many times as desired to define a patch of Spectres of any
size.  Figure~\ref{fig:superdupertiles} shows five levels of Spectre
clusters.  Notice that at each level of this construction, individual
tiles have the opposite handedness of those in the previous level.

As with the metatiles used in the analysis of the hat~\cite{hat}, these
substitution rules cannot be expressed as a set of rigid motions
and uniform scalings that pack tiles into supertiles.  This behaviour
is evident from the fact that the quadrilaterals in
Figure~\ref{fig:patch_construction} are not related by similarities.
However, also like metatiles, they quickly converge to limit shapes.

\FloatBarrier


%
%

\bibliographystyle{alphaurl}
\bibliography{tilings}

\begin{thebibliography}{SMKGS24}

\bibitem[AGS92]{AGS}
Robert Ammann, Branko Gr\"unbaum, and G.~C. Shephard.
\newblock Aperiodic tiles.
\newblock {\em Discrete Comput. Geom.}, 8:1--25, 1992.
\newblock \href {https://doi.org/10.1007/BF02293033}
  {\path{doi:10.1007/BF02293033}}.

\bibitem[Gar75]{gardnerconwaycriterion}
Martin Gardner.
\newblock More about tiling the plane: the possibilities of polyominoes,
  polyiamonds, and polyhexes.
\newblock {\em Scientific American}, 233(2):112--115, 1975.
\newblock URL: \url{https://www.jstor.org/stable/24949870}.

\bibitem[GBN91]{GiraultBeauquierNivat}
D.~Girault-Beauquier and M.~Nivat.
\newblock Tiling the plane with one tile.
\newblock In {\em Topology and category theory in computer science ({O}xford,
  1989)}, Oxford Sci. Publ., pages 291--333. Oxford Univ. Press, New York,
  1991.

\bibitem[GK23]{greenfeld2023tiling}
Rachel Greenfeld and Mihail~N. Kolountzakis.
\newblock Tiling, spectrality and aperiodicity of connected sets.
\newblock 2023.
\newblock \href {https://arxiv.org/abs/2305.14028} {\path{arXiv:2305.14028}}.

\bibitem[Gre07]{greenberg}
Marvin~J. Greenberg.
\newblock {\em Euclidean and Non-Euclidean Geometries: Development and
  History}.
\newblock W.H. Freeman, fourth edition, 2007.

\bibitem[GS99]{aht}
Chaim Goodman-Strauss.
\newblock Aperiodic hierarchical tilings.
\newblock In {\em Foams and emulsions ({C}arg\`ese, 1997)}, volume 354 of {\em
  NATO Adv. Sci. Inst. Ser. E: Appl. Sci.}, pages 481--496. Kluwer Acad. Publ.,
  Dordrecht, 1999.
\newblock \href {https://doi.org/10.1007/978-94-015-9157-7_28}
  {\path{doi:10.1007/978-94-015-9157-7_28}}.

\bibitem[GS09]{regprod}
Chaim Goodman-Strauss.
\newblock Regular production systems and triangle tilings.
\newblock {\em Theoret. Comput. Sci.}, 410(16):1534--1549, 2009.
\newblock \href {https://doi.org/10.1016/j.tcs.2008.12.012}
  {\path{doi:10.1016/j.tcs.2008.12.012}}.

\bibitem[GS16]{GS}
Branko Gr\"unbaum and G.C. Shephard.
\newblock {\em Tilings and Patterns}.
\newblock Dover, second edition, 2016.

\bibitem[GT23]{GT1}
Rachel Greenfeld and Terence Tao.
\newblock Undecidable {T}ranslational {T}ilings with {O}nly {T}wo {T}iles, or
  {O}ne {N}onabelian {T}ile.
\newblock {\em Discrete Comput. Geom.}, 70(4):1652--1706, 2023.
\newblock \href {https://doi.org/10.1007/s00454-022-00426-4}
  {\path{doi:10.1007/s00454-022-00426-4}}.

\bibitem[GT24]{GT2}
Rachel Greenfeld and Terence Tao.
\newblock A counterexample to the periodic tiling conjecture.
\newblock {\em Ann. of Math. (2)}, 200(1):301--363, 2024.
\newblock \href {https://doi.org/10.4007/annals.2024.200.1.5}
  {\path{doi:10.4007/annals.2024.200.1.5}}.

\bibitem[Hop26]{hopf}
H.~Hopf.
\newblock Zum {Clifford}-{Kleinschen} {Raumproblem}.
\newblock {\em Math. Ann.}, 95:313--339, 1926.
\newblock \href {https://doi.org/10.1007/BF01206614}
  {\path{doi:10.1007/BF01206614}}.

\bibitem[JR21]{JeandelRao}
Emmanuel Jeandel and Micha\"{e}l Rao.
\newblock An aperiodic set of 11 {W}ang tiles.
\newblock {\em Adv. Comb.}, (1):1--37, 2021.
\newblock \href {https://doi.org/10.19086/aic.18614}
  {\path{doi:10.19086/aic.18614}}.

\bibitem[Kap22]{Kaplan}
Craig~S. Kaplan.
\newblock Heesch numbers of unmarked polyforms.
\newblock {\em Contributions to Discrete Mathematics}, 17(2):150--171, 2022.
\newblock \href {https://doi.org/10.55016/ojs/cdm.v17i2.72886}
  {\path{doi:10.55016/ojs/cdm.v17i2.72886}}.

\bibitem[Ken92]{Kenyon}
Richard Kenyon.
\newblock Rigidity of planar tilings.
\newblock {\em Invent. Math.}, 107:637--651, 1992.
\newblock \href {https://doi.org/10.1007/BF01231905}
  {\path{doi:10.1007/BF01231905}}.

\bibitem[Ken93]{Kenyonerratum}
Richard Kenyon.
\newblock Erratum: ``{R}igidity of planar tilings''.
\newblock {\em Invent. Math.}, 112(1):223, 1993.
\newblock \href {https://doi.org/10.1007/BF01232432}
  {\path{doi:10.1007/BF01232432}}.

\bibitem[Ken96]{Kenyon2}
Richard Kenyon.
\newblock A group of paths in {$\mathbb{R}^2$}.
\newblock {\em Trans. Amer. Math. Soc.}, 348(8):3155--3172, 1996.
\newblock \href {https://doi.org/10.1090/S0002-9947-96-01562-0}
  {\path{doi:10.1090/S0002-9947-96-01562-0}}.

\bibitem[Mye19]{Myers}
Joseph Myers.
\newblock Polyomino, polyhex and polyiamond tiling, 2000--2019.
\newblock Accessed: February 19th, 2023.
\newblock URL: \url{https://www.polyomino.org.uk/mathematics/polyform-tiling/}.

\bibitem[Sch80]{conwaycriterion}
Doris Schattschneider.
\newblock Will it tile? {T}ry the {C}onway criterion!
\newblock {\em Math. Mag.}, 53(4):224--233, 1980.
\newblock \href {https://doi.org/10.2307/2689617} {\path{doi:10.2307/2689617}}.

\bibitem[SF88]{shapingspace}
Marjorie Senechal and George Fleck, editors.
\newblock {\em Shaping space: A polyhedral approach}, Design Science
  Collection. Birkh\"{a}user Boston, Inc., Boston, MA, 1988.

\bibitem[SMKGS24]{hat}
David Smith, Joseph~Samuel Myers, Craig~S. Kaplan, and Chaim Goodman-Strauss.
\newblock An aperiodic monotile.
\newblock {\em Comb. Theory}, 4(1):Paper No. 6, 91, 2024.
\newblock \href {https://doi.org/10.5070/C64163843}
  {\path{doi:10.5070/C64163843}}.

\bibitem[ST11]{ST1}
Joshua E.~S. Socolar and Joan~M. Taylor.
\newblock An aperiodic hexagonal tile.
\newblock {\em J. Combin. Theory Ser. A}, 118(8):2207--2231, 2011.
\newblock \href {https://doi.org/10.1016/j.jcta.2011.05.001}
  {\path{doi:10.1016/j.jcta.2011.05.001}}.

\bibitem[ST12]{ST2}
Joshua E.~S. Socolar and Joan~M. Taylor.
\newblock Forcing nonperiodicity with a single tile.
\newblock {\em Math. Intelligencer}, 34(1):18--28, 2012.
\newblock \href {https://doi.org/10.1007/s00283-011-9255-y}
  {\path{doi:10.1007/s00283-011-9255-y}}.

\end{thebibliography}

\end{document}